\documentclass[a4paper,11pt]{amsart}
\usepackage[dvipsnames]{color}
\usepackage{amsfonts,amssymb,amsmath,amscd,amstext}
\usepackage[colorlinks=true,linkcolor=blue,citecolor=blue]{hyperref}
\usepackage[utf8]{inputenc}
\usepackage{graphicx}
\usepackage{changes}
\usepackage{comment}
\usepackage{cleveref}
\usepackage{mathdots}
\usepackage[a4paper,top=2.5cm,bottom=2.5cm,left=2cm,right=2cm]{geometry}
\newtheorem{theorem}{Theorem}[section]
\newtheorem{proposition}[theorem]{Proposition}
\newtheorem{corollary}[theorem]{Corollary}

\newtheorem{lemma}[theorem]{Lemma}
\newtheorem{definition}[theorem]{Definition}

\theoremstyle{remark}
\newtheorem{remark}[theorem]{Remark}

\newcommand{\Rn}{\mathbb R^n}
\newcommand{\Rm}{\mathbb R^m}
\newcommand{\R}{\mathbb R}
\newcommand{\Q}{\mathbb Q}
\renewcommand{\L}[1]{\mathcal L^{#1}}
\newcommand{\G}{\mathbb G}

\newcommand{\eps}{\epsilon}

\newcommand{\vf}{\varphi}

\renewcommand{\H}{\mathbb H}

\newcommand{\diam}{{\rm diam}}

\newcommand{\average}{{\mathchoice {\kern1ex\vcenter{\hrule height.4pt
width 6pt depth0pt} \kern-9.7pt} {\kern1ex\vcenter{\hrule height.4pt width
4.3pt depth0pt}
\kern-7pt} {} {} }}
\newcommand{\ave}{\average\int}

\newcommand{\dist}{\mbox{dist}}
\newcommand{\N}{\mathbb N}

\renewcommand{\rho}{\varrho}
\renewcommand{\epsilon}{\varepsilon}

\renewcommand{\theta}{\vartheta}

\newcommand{\Z}{\mathbb Z}

\renewcommand{\L}[1]{\mathcal L^{#1}}

\newcommand{\gl}{\lambda}

\newcommand{\spt}{{\mathrm{spt}}}
\newcommand{\ci}{{\mathbf C}}
\newcommand{\leb}[1]{\mathcal L^{#1}}

\newcommand\blfootnote[1]{
    \begingroup
    \renewcommand\thefootnote{}\footnote{#1}
    \addtocounter{footnote}{-1}
    \endgroup
}
\begin{document}

\title
[Intrinsic differentiability and Area formula]{On some intrinsic differentiability properties for Absolutely Continuous functions between Carnot groups and the Area formula
}


\author[A.~Pinamonti]{Andrea Pinamonti}
\address{Dipartimento di Matematica\\ Universit\'a di Trento\\
Via Sommarive, 14, 38123 Povo TN, Italy}
\email{andrea.pinamonti@unitn.it}

\author [F. Serra Cassano]{Francesco Serra Cassano}
\address{Dipartimento di Matematica\\ Universit\'a di Trento\\
Via Sommarive, 14, 38123 Povo TN, Italy}
\email{francesco.serracassano@unitn.it}
\author [K. Zambanini]{Kilian Zambanini}
\address{Dipartimento di Matematica\\ Universit\'a di Trento\\
Via Sommarive, 14, 38123 Povo TN, Italy}
\email{kilian.zambanini@unitn.it}


\begin{abstract}
We discuss $Q-$absolutely continuous functions between Carnot groups, following Mal\'y's definition for maps of several variables (\cite{Maly}). Such maps enjoy nice regularity properties, like continuity, Pansu differentiability a.e., weak differentiability and an area formula. Furthermore, we extend Stein's result concerning the sharp condition for continuity and differentiability a.e. of a Sobolev map in terms of the integrability of the weak gradient: more precisely, we prove that a Sobolev map between Carnot groups with horizontal gradient of its sections uniformly bounded in $L^{Q,1}$ admits a representative which is $Q$-absolutely continuous.
\end{abstract}
\maketitle
\begin{flushright}
{\it{In Memory of Jan Mal\'y}}
\end{flushright}

\section{Introduction}
In this paper we  study Sobolev regularity, continuity and Pansu differentiability  for  maps \linebreak $f:\Omega\subset\G_1\to\G_2$ satisfying suitable intrinsic notions of absolute continuity or bounded variation, called {\it $Q$-absolute continuity} and {\it bounded $Q$-variation},  where $\G_1$ and $\G_2$ are Carnot groups and $\Omega$ is an open set in $\G_1$. In particular, we will apply such a study for extending in this setting an analogous of the remarkable Euclidean Stein's result \cite{SteinAnnMath} concerning the sharp condition providing  continuity and differentiability a.e. in terms of integrability of the weak gradient, thus paving the way to the study of the gradient regularity of solutions to the sub-Laplace Poisson equation, see e.g. \cite{KM}.
Moreover we will prove an area formula for $Q$-absolutely continuous maps between Carnot group.

The definition and main properties of  a Carnot group are collected in section \ref{CG}. Let us quickly recall here that a \textit{Carnot group} $\G$ is a simply connected nilpotent Lie group which can be represented as $\G\equiv\,(\Rn,\cdot)$ equipped with: \begin{itemize}
    \item 
a family of \textit{intrinsic dilations} $(\delta_\lambda)_{\lambda>0}\colon\G\to\G$, which are automorphisms of $\G$;\item a subbundle $H\G$ of the tangent bundle $T\G$, called {\it horizontal subbundle}, which generates by commutation the whole $T\G$;\item a {\it Carnot-Carath\'eodory} (or, also,  {\it control}) {\it metric} $d_c$ associated to the horizontal subbundle $H\G$, which is left-invariant and homogeneous with respect to the intrinsic dilations. \end{itemize}The Haar measure of $\G$ turns out to be the $n$-dimensional Lebesgue measure $\mathcal L^n$ on $\Rn$ and the  metric (Hausdorff) dimension of $(\G,d_c)$ coincides with its homogeneous dimension $Q$. If $\G$ is commutative, then $\G\equiv\Rn$ as Euclidean vector space; if $\G$ is not commutative, then $Q>\,n$. In particular, we will see that the homogeneous dimension $Q$ will play the same role of the topological dimension $n$ in the Euclidean setting.

A Carnot group is  a sub-Riemannian structure, where several classical differential notions and results of the Euclidean/Riemannian geometry do not apply (see\cite{gromov1,gromov2}). Thus new intrinsic notions are needed in this setting, such as, for instance, the ones of  differentiable and Sobolev regular function between Carnot groups (see Definitions \ref{pansudifferentiability}  and \ref{SobspacesvaluedCarnot}, respectively). Nevertheless, from the intrinsic point of view, some important analogous results to Euclidean setting still hold  for maps  between Carnot groups, such as a Rademacher differentiability result due to Pansu (see Theorem \ref{pansuthm}) and a Sobolev-Morrey  embedding theorem (see,  Theorem \ref{Sobembp>Q}).

Let us  stress that the  intrinsic notion of differentiability  for mappings between Carnot groups, and the related results,  cannot be always recovered within the  theory of differentiation  for mappings between metric spaces (see Section \ref{DiffMS} ).

The notions  of {\it $n$-dimensional  absolutely continuous  function} and {\it $n$-variation}  for a map \linebreak $f:\,\Omega\subset\Rn\to\Rm$ were initially introduced by Mal\'y in \cite{Maly}. They have been studied in several papers (see, for instance, \cite{Cso, Hencl1, Hencl2,KKM}). An historical account as far as the concept of absolute continuity  for functions can be found in \cite{Maly} and the references therein.

These notions can also be introduced in a metric measure setting, thus it is immediate to extend them to maps between Carnot groups. In the following  the balls involved are the ones with respect to the CC distance $d_c$. 
\begin{definition}\label{abscont} Let $\G_1=\,(\Rn,\cdot,d_1)$ and $\G_2=\,(\Rm,*,d_2)$ be two Carnot groups, let $Q$ be the homogeneous dimension of $\G_1$ (see (\ref{homdim})) and let  $f:\,\Omega\subset\G_1\to\G_2$, with  $\Omega$  open set in $\G_1$. 
\begin{itemize}
\item[(i)] $f$ is said to be {\it $Q$-absolutely continuous} if for each $\eps>0$ there exists $\delta=\,\delta(\eps)>\,0$ such that, for each finite disjoint family of open balls $U_1,\dots,U_k$ contained in $\Omega$,
\[
\text{if }\sum_{i=1}^k\leb n (U_i)<\delta\text{ then }\sum_{i=1}^k\left({\rm osc}_{U_i}f\right)^Q<\,\eps
\]
where
\[
{\rm osc}_{U}f:=\,{\rm diam}_{\G_2}(f(U))\,.
\]
\item[(ii)] The {\it $Q$-variation} of $f$ is defined as
\[
V_Q(f,\Omega):=\,\sup\left\{\sum_{i=1}^k\left({\rm osc}_{U_i}f\right)^Q:\{U_1,\dots,U_k\} \text{ disjoint family of balls in }\Omega\,\right\}\,.
\]
We say that $f$ has bounded $Q$-variation if  $V_Q(f,\Omega)<\,\infty$.
\end{itemize}
\end{definition}
We will denote by $BV^Q(\Omega,\G_2)$ the class of functions $f:\,\Omega\subset\G_1\to\G_2$ with bounded $Q$-variation and with
$AC^Q(\Omega,\G_2)$ the class of $Q$-absolutely continuous functions. We use the notation $BV_{\rm loc}^Q(\Omega,\G_2)$ for functions which are locally of bounded $Q$-variation and $AC_{\rm loc}^Q(\Omega,\G_2)$ for the family of functions which are locally $Q$-absolutely continuous.  It actually holds that  $AC^Q(\Omega.\G_2)\subset BV^Q(\Omega,\G_2)$ if $\Omega\subset\G_1$ is a bounded open set (see Proposition \ref{ACsubsetBV}).\par

The study of absolutely continuous functions has been extended from the Euclidean framework to the one of metric measures spaces: firstly, the study of  real-valued functions defined on measure metric spaces (see, for instance, \cite{RaMo,Romanov}), and then the one of metric space-valued maps 
(see, for instance, \cite{WZ,WZ2,MaZ,LZ}). We are going to compare later our results with the previous ones.  \par

Note that also in this context the homogeneous dimension $Q$ of the  Carnot group $\G_1$ plays the same role of the Euclidean dimension. Indeed, we will see that, roughly speaking, the class $AC_{\rm{loc}}^Q(\Omega,\G_2)$ locally lies between the  class of Sobolev maps  which admit weak horizontal gradient with summability  $Q$, that is the class $W_{\rm loc}^{1,Q}(\Omega,\G_2)$  (see Definition 
\ref{SobspacesvaluedCarnot}), and the union over $p>Q$ of the spaces $W_{\rm loc}^{1,p}(\Omega,\G_2)$. Moreover it shares good properties with the maps in $W_{\rm loc}^{1,p}(\Omega,\G_2)$ with $p>Q$.
\par
 In section \ref{ACF} we begin our study by stressing some simple classes of functions contained in $AC^Q(\Omega,\G_2)$. In particular we will see that, if $f$ satisfies the so-called (RR) condition, then $f\in AC^Q(\Omega,\G_2)$ and we will show that, if $f$ is a Sobolev map with exponent $p>Q$, that is $f\in W^{1,p}(\Omega,\G_2)$, then $f\in AC^Q(\Omega,\G_2)$ (see Theorem \ref{p>Qdiff}). Moreover we will prove that maps with finite $Q$-variation between Carnot groups, that is maps in $BV^Q(\Omega,\G_2)$,  are  Pansu-differentiable a.e. and admit weak horizontal gradient (see Theorems \ref{diffthm} and \ref{Weakdiff}). As a consequence, since $AC^Q_{\rm loc}(\Omega,\G_2)\subset BV^Q_{\rm loc}(\Omega,\G_2)$, also  functions in $AC^Q(\Omega,\G_2)$ share the same properties.\par
In section \ref{Steinresult} we will prove that a Sobolev function $f:\Omega\subset\G_1\to\G_2$  admitting weak gradient in the Lorenz space $L^{Q,1}(\Omega)$ admits a representative in $AC^Q(\Omega,\G_2)$ (Theorem \ref{SobmapRR}), extending in this way Stein's result on Euclidean spaces. Moreover we will discuss this result also for maps in a suitable class of Orlicz-Sobolev spaces between Carnot groups (see Appendix \ref{appendix6} and in particular Theorem \ref{finale}).
\par

In section \ref{AF} we will show an area formula for functions in $AC^Q(\Omega,\G_2)$ (see Theorem \ref{areaformulaAC} and Corollary \ref{areaformula2}), which extends the one proved independently in \cite{Ma2001} and in \cite{Vo00} for Lipschitz functions.
\par
Eventually in Appendix \ref{DiffMS} we will collect, for sake of completeness,  some relevant notions and results about the differentiability and Sobolev regularity  of metric spaces-valued maps defined on a metric measure space, in order to compare our results with the previuos ones.

\vspace{5pt}

\textbf{Acknowledgments.}
The authors are members of {\em Gruppo Nazionale per l'Analisi Matematica, la Probabilit\`a e le loro Applicazioni} (GNAMPA), of the {\em Istituto Nazionale di Alta Matematica} (INdAM) and they are partially funded by the European Union under NextGenerationEU. PRIN 2022 Prot. n. 2022F4F2LH.
The authors would like to thank B. Franchi, R. Serapioni and V. Magnani for very useful discussions on the topics of the present paper.  We would like to thank the referee for carefully reading the paper and for giving such constructive comments which substantially helped improving the quality of the paper.

\section{Preliminaries}\label{CG}
%

\subsection{Carnot groups}

A \textit{Carnot group} $(\G,\cdot)$ is a simply connected and nilpotent Lie group whose Lie algebra $\mathfrak{g}$ of left-invariant vector fields has dimension $n$ and admits a stratification of step $s\in\N$, that is
\begin{equation*}
\mathfrak g=V_1\oplus V_2\oplus\cdots\oplus V_s,
\end{equation*}
where the vector spaces $V_1,\dots,V_s\subset\mathfrak g$ satisfy
\begin{equation*}
V_i=[V_1,V_{i-1}]\quad \text{for}\ i=2,\ldots,s, \quad [V_1,V_s]=\{0\}. 
\end{equation*}
We set $m_i=\dim(V_i)$ and $h_i=\sum_{j=1}^i m_j$ for $i=1,\dots,s$. 
We fix an \emph{adapted basis} of $\mathfrak{g}$, i.e.\ a basis $X_1,\dots,X_n$ such that
\begin{equation*}
X_{h_{i-1}+1},\dots,X_{h_i}\ \text{is a basis of}\ V_i,\quad i=1,\dots,s.
\end{equation*}
The subspace $V_1$ of $\mathfrak g$ is also called \textit{horizontal layer}, while $X_1,\dots,X_{m_1}$ are the \textit{generating vector fields} of the group.
We endow the algebra $\mathfrak{g}$ with the left-invariant Riemannian metric $\langle{\cdot,\cdot}\rangle$ that makes the basis $X_1,\dots,X_n$ orthonormal.
Exploiting the exponential identification $p=\exp\left(\sum_{i=1}^np_iX_i\right)$, we can identify $\G$ with $\R^n$, endowed with the group law determined by the Campbell--Hausdorff formula.
In particular, the identity $e\in\G$ corresponds to $0\in\R^n$ and the \emph{inversion map} becomes $\iota(p)=p^{-1}=-p$ for any $p\in\G$.
Moreover, it is not restrictive to assume that $X_i(0)=\mathrm{e}_i$ for any $i=1,\dots,n$. 
Therefore, by left-invariance, we get
\begin{equation*}
X_i(p)=d\tau_p\mathrm{e}_i, \quad i=1,\dots,n,
\end{equation*}
where $\tau_p\colon\G\to\G$ is the \emph{left-translation} by $p\in\G$, i.e.\ $\tau_p(q)=p\cdot q$ for any $q\in\G$. \\
The subbundle of the tangent bundle $T\G$ that is spanned by the
vector fields $X_1,\dots,X_{m_1}$ plays a particularly important
role in the theory and it is called the {\it horizontal bundle}
$H\G$; the fibers of $H\G$ are $$ H\G_x=\mbox{span
}\{X_1(x),\dots,X_{m_1}(x)\},\qquad x\in\G .$$ 
For any $i=1,\dots,n$, the \emph{degree} $d(i)\in\{1,\dots,s\}$ of the basis vector field $X_i$ is $d(i)=j$ if and only if $X_i\in V_j$. 
The group \emph{dilations} $(\delta_\lambda)_{\lambda\ge0}\colon\G\to\G$ are hence given by
\begin{equation*}
\delta_\lambda(p)=\delta_\lambda(p_1,\dots,p_n)
=(\lambda p_1,\dots,\lambda^{d(i)} p_i,\dots,\lambda^s p_n) 
\quad
\text{for all}\ p\in\G.
\end{equation*}
The Haar measure of the group $\G$ coincides with the $n$-dimensional Lebesgue measure~$\mathcal L^n$.
The homogeneity property $\mathcal L^{n} 
 (\delta_\lambda(E))=\lambda^Q\mathcal L^n (E)$ holds for any measurable set $E\subset\G$, where 
 \begin{equation}\label{homdim}
 Q:= \sum_{i=1}^s i\dim(V_i)\in\N  
 \end{equation} 
is the \emph{homogeneous dimension} of $\G$.
For notational convenience, we will use also the shorthand $|E|=\mathcal{L}^n(E)$.\\
In the following proposition (see \cite{FSSCstep2}) we collect some fairly
elementary properties of the group operation.
\begin{proposition}\label{legge di gruppo}
The group product has the form
\begin{equation} x\cdot y=x+y+\mathcal Q(x,y),\qquad
\forall x,y\in\Rn \label{legge di gruppo1}
\end{equation}
where $\mathcal Q=(\mathcal Q_1,\dots,\mathcal
Q_n):\Rn\times\Rn\to\Rn$ and each $\mathcal Q_i$ is a homogeneous
polynomial of degree $d(i)$ i.e.
\begin{equation*}\mathcal Q_i(\delta_\lambda x,\delta_\lambda
y)=\lambda^{d(i)}\mathcal Q_i(x,y),\quad \forall x,y\in \G.
\end{equation*} 
Moreover $\mathcal Q$ is anti-symmetric, which means
\begin{equation*}
\mathcal Q_j(x,y) = - \mathcal Q_j(-y,-x),
\end{equation*}
and also
\begin{gather}
\mathcal Q_1(x,y)=...=\mathcal Q_{m_1}(x,y)=0,\notag\\
\label{legge di gruppo3} \mathcal Q_j(x,0)=\mathcal
Q_j(0,y)=0,\\
\nonumber\mathcal Q_j(x,x)=\mathcal Q_j(x,-x)=\mathcal Q_j(-x,x)=0,\quad
\text{for}\;\;m_1< j\leq n,\\\nonumber \mathcal Q_j(x,y)=\mathcal
Q_j(x_1,\dots,x_{h_{i-1}},y_1,\dots,y_{h_{i-1}}),
\quad\text{if}\quad 1<i\leq s\quad\text{and}\quad   j\leq
h_i.\label{legge di gruppo2}
\end{gather}


\end{proposition}

\begin{definition}\label{horizontalcurve}
An absolutely continuous curve $\gamma\colon [a,b]\to \mathbb{G}$ is horizontal if there exist $u_{1}, \dots, u_{m_1}\in L^{1}[a,b]$ such that \[
\gamma'(t)=\sum_{j=1}^{m_1}u_{j}(t)X_{j}(\gamma(t)) \quad \mbox{for a.e.}\  t\in [a,b].
\]
We define the \emph{horizontal length} of a horizontal curve $\gamma$ by $L(\gamma)=\int_{a}^{b}|u(t)|dt$, where $u=(u_{1}, \dots, u_{m_1})$ and $|\cdot|$ denotes the Euclidean norm on $\mathbb{R}^{m_1}$.
\end{definition}

Chow-Rashevskii Theorem asserts that any two points of $\mathbb{G}$ can be connected by horizontal curves \cite[Theorem 9.1.3]{BLU}. Actually, any two points can be connected by a geodesic, despite such a curve can be not unique \cite[Theorem 1.4.4]{monti}. This allows us to define the \emph{Carnot-Carath\'eodory distance (CC distance)} as follows.

\begin{definition}
The \emph{CC distance} between $x,y\in \mathbb{G}$ is
\[d_c(x,y):=\inf \{ L(\gamma) |\, \gamma \colon [0,1]\to \mathbb{G} \mbox{ horizontal joining }x\mbox{ to }y \}.\]
\end{definition}

This CC metric induces the Euclidean topology and it is an \textit {invariant distance}, since it satisfies $d_c(zx,zy)=d_c(x,y)$ (left translation invariance) and $d_c(\delta_{r}(x),\delta_{r}(y))=rd_c(x,y)$ (homogeneity) for $x, y, z\in \mathbb{G}$ and $r>0$. Every invariant distance $d$ induces a \textit{homogeneous norm}, simply defining $\|x\|:=d(x,0)$.
We denote the CC open ball with center $x$ and radius $r$ by $U_r(x)$ or by $U(x,r)$ and in particular we use the shorthand $U_1=U(0,1)$. When convenient, we will consider the closed CC ball of center $x$ and radius $r$, denoted by $B(x,r)$. It can be proved that $\mathcal{L}^n(U_r(x))=r^Q \mathcal{L}^n(U_1(0))$, where $Q$ is the homogeneous dimension of $\G$ (see (\ref{homdim})). Therefore $Q$ turns out to be the Hausdorff dimension of the metric measure space $(\G,d_c,\mathcal L^n)$. In particular note that, if the step of $\G$ is $s\ge\,2$, then $Q>\,n$, which means that the Hausdorff metric dimension of $\G$ is larger than its topological dimension $n$.

Following~\cite[Thm. 5.1]{FSSCstep2}, we also introduce the left-invariant and homogeneous distance $d_\infty(p,q)=d_\infty(q^{-1}\cdot p,0)$ for $p,q\in\G$, where, identifying $\G=\R^n=\R^{m_1}\times\dots\times\R^{m_s}$ and letting $\pi_{\R^{m_i}}\colon\R^n\to\R^{m_i}$ to be the canonical projection for $i=1,\dots,s$, 
\begin{equation}
\label{eq:distance}
d_\infty(p,0)
=
\max\{\epsilon_i|\pi_{\R^{m_i}}(p)|_{\R^{m_i}}^{1/i} : i=1,\dots,s\}
\quad
\text{for all}\ p\in\G
,
\end{equation}
with constants $\epsilon_1=1$ and $\epsilon_i\in(0,1)$ for all $i=2,\dots,s$ depending on the structure of $\G$.
We use the shorthand $\|p\|_\infty=d_\infty(p,0)$ for $p\in\G$.
Consequently, for $p\in\G$ and $r>0$, we define the open and closed balls with respect to this metric as
\begin{equation*}
U_{\infty}(p,r)
=
\{q\in\G : d_{\infty}(q,p)<r\},
\qquad
 B_{\infty}(p,r)
=
\{q\in\G : d_{\infty}(q,p)\le r\}.
\end{equation*}

\subsection{Differential Calculus on Carnot Groups}
This section is taken from \cite{FSSCstep2} and \cite[Chapter 3]{magnani thesis}; we refer to \cite{Ma3} for a generalization of these contents.
The notion of \emph {P-differentiability} for functions acting between Carnot groups was introduced by Pansu in \cite{pansu}.
\begin{definition}\label{Pdiff}
  Let $\G_1, \G_2$ be Carnot groups, with homogeneous
  norms $\|\cdot\|_1$, $\|\cdot\|_2$ and dilations $\delta^1_\lambda$, $\delta^2_\lambda$. We say that 
  $L:\G_1\to\G_2$ is
   \emph{H-linear}, or
  is a \emph {homogeneous homomorphism}, if $L$ is a group
   homomorphism  such that 
$$
   L(\delta^1_\lambda g)=\delta^2_\lambda L(g), \quad \text{for all  }g\in \G_1 \text{ and } \gl>0.
$$
\end{definition}

\begin{definition}\label{pansudifferentiability} Let $\Omega\subset\G_1$ be open.
We say that $f:\Omega\to\G_2$ is \emph {P-differentiable} at $x_0\in\Omega$ if
  there is a $H$-linear function $d_Pf(x_0):\G_1\rightarrow\G_2$ such that
\begin{equation*}
  \|\left(d_Pf(x_0)(x_0^{-1}\cdot
  x)\right)^{-1}\cdot f(x_0)^{-1}\cdot f(x)\|_2= o\big( \| x_0^{-1}\cdot x\|_1\big),\quad  \text {as $\|x_0^{-1}\cdot x\|_1\to 0$}
\end{equation*}
 where $o(t)/t\to 0$ as $t\to 0^+$.
The $H$-linear map $d_Pf(x_0)$ is called  the \emph{Pansu's differential} of $f$ at $x_0$.
\end{definition}
The fundamental result where P-differentiability applies is the following generalization of Rademacher theorem and it is due to P. Pansu \cite{pansu}. Other possible proofs of this result were obtained by S.K. Vodopyanov \cite{Vo00} and by V. Magnani \cite{Ma2001}, see also \cite{monti}. Pansu's theorem was later extended by Magnani in \cite{Ma3}, in which a general homogeneous group as target space is considered. In \cite{MR} the authors proved the result for Banach space valued mappings (with the Radon-Nikodym property on the first layer); see also the subsequent work \cite{MPS}.

\begin{theorem}[Pansu \cite{pansu}]\label{pansuthm} Let $\G_1=(\R^n,\cdot)$ and $\G_2=(\R^m,*)$ be two Carnot groups equipped with two CC metrics denoted, respectively, by $d_1$ and $d_2$. Let $f:\,\Omega\subset \G_1 \to \G_2$  be a Lipschitz continuous function with $\Omega$ an open set. Then $f$ is P-differentiable at $\mathcal L^n$-a.e. in $\Omega$ .
\end{theorem}
A version of the classical Stepanov's theorem \cite{Stepanov} holds for maps between Carnot groups, which extends Pansu-Rademacher theorem (see \cite[Theorem 3.1]{Vo00}).
\begin{theorem}[Stepanov's theorem for Carnot groups]\label{Stepanov} Let $\G_1=(\R^n,\cdot)$ and $\G_2=(\R^m,*)$ be two Carnot groups equipped with two CC metrics denoted, respectively, by $d_1$ and $d_2$. Let $f:\,\Omega\subset \G_1\to \G_2$  and assume that
\begin{equation}\label{puLcf}
{\rm Lip}(f)(p):=\,\limsup_{q\to p}\frac{d_2(f(q),f(p))}{d_1(q,p)}<\infty \text{ for $\leb n$-a.e. }p\in\Omega.
\end{equation}
Then $f$ is P-differentiable at $\mathcal L^n$-a.e. $p\in\Omega$ .
\end{theorem}
Let $X_1,\dots,X_{m_1}$ be a generating family of vector fields. For $j=1,\dots,m_1$, we say that $f:\G\to\R$ is {\it differentiable along} $X_j$ at a point $x_0$ if the map $\lambda\mapsto f(\tau_{x_0}
(\delta_{\lambda} e_j))\,=f(\exp(\lambda X_j)(x_0))$
is differentiable at $\lambda =0$, where $e_j$ is the
$j$-th vector of the canonical basis of $\Rn$.
In other words, $f$ is differentiable along $X_j$ at $x_0$ if there exists
\[X_jf(x_0):=\lim_{t\to 0}\frac{f(\exp(tX_j)(x_0))-f(x_0)}{t}.\]
Even if $f$ is not differentiable, we can still define the discrete difference quotients:
\begin{definition}\label{diff.quo}
Let $f:\Omega\to \mathbb{R},\, i=1,\dots,m_1$. We define the \textit{difference quotient} at $x$ in direction $X_i$ as
    \[D^h_if(x):=\frac{f(\exp(hX_i)(x))-f(x)}{h}\]
    defined for $x\in\Omega'\subset\subset\Omega$ with $0<|h|<\dist(\Omega', \partial \Omega)$. \end{definition}
Notice that, if $f:\Omega\subset\G\to\R$ is P-differentiable at a point $x_0$, then it is also differentiable along $X_j$ at $x_0$ and 
\[X_jf(x_0)= d_Pf(x_0)(e_j).\]
 For any function $f:\G\to \R$ for which the partial derivatives
$X_jf$ exist, we
define the \textit{horizontal gradient} of $f$, denoted by
$\nabla_{\G}f$, as 
\begin{equation*}
\nabla_{\G}f:=\sum_{i=1}^{m_1}(X_if)X_i.
\end{equation*}
We will represent the horizontal gradient in coordinates $\nabla_\mathbb Gf=(X_1f,\dots,X_{m_1}f).$
We conclude this section by recalling the notion of convolution in Carnot groups; we address the interested reader to \cite[Proposition 1.20]{FS}
for the proof.
\begin{lemma}\label{lemmacal}
Let $\G=(\Rn,\cdot)$ be a Carnot group and let $\rho\in \ci^{\infty}(\Rn )$ be such that $0\le\rho\le 1$, $\int_{\Rn}\rho \,d\mathcal L^n=1$, $\spt\rho\subset U_1$ and $\rho(x^{-1})=\rho(-x)=\rho(x)$ for all $x\in\Rn$. Let us denote 
\begin{equation}\label{mollifier}
\rho_{\eps}(x):=\eps^{-Q}\rho\left(\delta_{1/\eps}(x)\right)\qquad x\in\Rn\,;
\end{equation}
\begin{equation}\label{convolmollifier}
(f\star \rho_{\eps})(x):=\int_{\Rn} \rho_{\eps}(x\cdot\, y^{-1})\,f(y)\,d\mathcal L^n(y)=\int_{\Rn} \rho_{\eps}(y)\,f(y^{-1}\cdot\,x)\,d\mathcal L^n(y)
\end{equation}
Then
\begin{itemize}
\item[(i)] if $f\in L^p(\Rn),\ 1\le p<\infty$, then $f\star \rho_{\eps}\in \ci^{\infty}(\Rn)$ and $f\star \rho_{\eps}\to f$ in $L^p(\Rn)$ as $\eps\to 0$;
\item[(ii)] $\spt\ f\star \rho_{\eps}\subset \spt\ f\cdot B(0,\eps)$;
\item[(iii)] $X_j(\vf\star \rho_{\eps})=X_j\vf\star \rho_{\eps}$ for any $\vf\in \ci_c^{\infty}(\Rn)$ and each $j=1,\dots,m_1$;
\item[(iv)] $\int_{\Rn}(f\star\rho_{\eps})\,g\,d\mathcal L^n= \int_{\Rn}(g\star\rho_{\eps})\,f\,d\mathcal L^n$ for every $f\in L^{\infty}(\Rn)$, $g\in L^{1}(\Rn)$;
\item[(v)] if $f\in L^{\infty}(\Rn)\cap \ci^0(\Omega)$ for a suitable open set $\Omega\subset\Rn$ then $f\star\rho_{\eps}\to f$ uniformly on compact subsets of $\Omega$ as $\eps\to 0$.
\end{itemize}
\end{lemma}

\subsection{Lorentz space} Let $\Omega\subset\G\equiv\Rn$ be an open set, let $Q$ denote the homogeneous dimension of $\G$ and let $A:=\,\leb{n}(\Omega)$. If $g:\Omega\to\bar\R$ is a measurable function, $\leb{n}$-a.e. finite, we define its {\it distribution function}  $\lambda_g: [0,\infty)\to [0,\infty]$ as
\[
\lambda_g(s):=\,\leb{n}\left(\left\{x\in\Omega:\,|g(x)|>\,s\right\}\right)
\]
and the {\it  nonincreasing rearrangement} $g^*:\,[0,\infty)\to [0,\infty]$ of $g$ as
\[
g^*(t):=\,\inf\left\{s\ge\,0:\,\lambda_g(s)\le\,t\right\}\,.
\]
It is well-known that $\lambda_g$ and $g^*$ are  nonincreasing and right-continuous.  Moreover $g^*$ is the nonincreasing function which is equimeasurable with $|g|$, namely 
\[\L n (\{x\in\Omega:|g(x)|>s\})=\L 1 (\{t\in\R^+:g^*(t)>s\})\qquad \text{for all }s>0. \]


The {\it Lorentz space} $L^{Q,1}(\Omega)$ is defined as the class of all measurable functions $g:\Omega\to\bar\R$, $\leb{n}$-a.e. finite,  for which the quasi-norm
\[
\|g\|_{L^{Q,1}(\Omega)}:=\,\int_0^A t^{\frac{1-Q}{Q}}\,g^*(t)\,dt
\]
is finite. Note that it is not a norm since the triangle inequality may fail. 
A systematic study of Lorentz spaces is carried out in \cite{Stein} whereas a simpler recent account can be found in  \cite{KKM}.\\
 Let us recall some properties of the Lorentz space $L^{Q,1}(\Omega)$ whose proof can be found in \cite{Bennett} and \cite{KKM}.
 \begin{proposition}\label{propLQ1}\phantom{1}
 \begin{itemize}
 \item[(i)] For any measurable function $g:\Omega\to\R$,
 \[ \|g\|_{L^{Q,1}(\Omega)}=\,Q\int_0^\infty\lambda_g(s)^{1/Q}\,ds\,.
 \]
  \item[(ii)] $L^{Q,1}(\Omega)\subset L^{Q}(\Omega)$; moreover, if $|\Omega|<\,\infty$, it holds that $L^p(\Omega)\subset L^{Q,1}(\Omega)$ for each $p>Q$.

 \item[(iii)] If $|g_1|\le\,|g_2|$ a.e. in $\Omega$, then
 \[
 \|g_1\|_{L^{Q,1}(\Omega)}\le\, \|g_2\|_{L^{Q,1}(\Omega)}\,.
 \]
 \end{itemize}
 \end{proposition}
 Notice that, in view of point (ii) of the previous Proposition, the Lorentz space $L^{Q,1}(\Omega)$ is, locally, an intermediate space between $L^Q(\Omega)$ and $L^p(\Omega)$ for all $p>Q$.\\
Let us also recall the following  remarkable characterization of $L^{Q,1}(\Omega) $ in terms of Orlicz integrability conditions, contained in \cite[Corollary 2.4]{KKM}, which we will need later.
\begin{theorem}\label{charLQ1} Let $g$ be a  measurable function on $\Omega$. Then the following
properties are equivalent:
\begin{itemize}
\item [(i)] $g\in L^{Q,1}(\Omega) $.
\item [(ii)] There is a positive nonincreasing function $\vf\in L^{1/Q}((0,\infty))$ such that

\begin{equation}
\int_\Omega F_\vf(|g|)\,dx<\,\infty\,,
\end{equation}
where
\begin{equation}\label{Fvf}
F_\vf(s):=
\begin{cases}
s\,\vf(s)^{\frac{1}{Q}-1}&\text{ if }s>\,0\\
0&\text{ if }s=\,0.
\end{cases} 
\end{equation}
\end{itemize}
\end{theorem}
\subsection{Sobolev spaces on Carnot groups.}
In this section we introduce and study $L^p$ and Sobolev spaces  for maps between Carnot groups. 
\subsubsection{Sobolev functions from a Carnot group to $\mathbb{R}$} Sobolev spaces for scalar functions defined on a Carnot group $\G$ are constructed in the usual way by considering distributional derivatives along the basis $X_1,\dots, X_{m_1}$ of the horizontal layer of $\mathbb G$:
     a function $g\in L^1_{loc}(\Omega)$ is the \textit{weak (or distributional) derivative of $f$ along $X_i$} if 
    \[\int_\Omega g\varphi\, dx=-\int_\Omega fX_i\varphi\,dx\qquad \text{for all }\varphi\in C^\infty_c(\Omega).\]
As usual we will write $X_if$ also to denote the weak derivative of $f$ along $X_i$ and $\nabla_\mathbb G f=(X_1f,\dots,X_{m_1}f)$ for the \textit{weak horizontal gradient}.
\begin{definition}\label{horSob}
Let $\Omega$ be an open subset of $\G\equiv\Rn$ and $1\leq p\leq\infty$.
We define the (horizontal) Sobolev space $W^{1,p}_{\mathbb{G}}(\Omega)$ as the space of all the functions with finite norm
\begin{equation}\label{defNorm}
\Vert u\Vert_{W^{1,p}_{\mathbb{G}}} := \Vert u\Vert_{p} + \Vert \nabla_{\G}u\Vert_{p},
\end{equation}
where $\nabla_{\G}u$ denotes the horizontal distributional gradient of $u$.
\end{definition}
The following result is well-known, see e.g. \cite{FS}, \cite{FSSC2}, \cite{garnie}.
\begin{proposition} For any $1\leq p\leq
\infty$,  $(W^{1,p}_{\G}(\Omega), \Vert \cdot\Vert_{W^{1,p}_{\mathbb{G}}})$ is a Banach space, reflexive if $1<p<\infty$.
Moreover, $W^{1,2}_{\G}(\Omega)$ is a Hilbert space.
\end{proposition}

Another way to define the $W^{1,p}_{\mathbb{G}}$ with $1\leq p<\infty$
is to take the closure of
$C^\infty$ functions w.r.t the norm \eqref{defNorm}. As in the Euclidean
case, the two approaches are equivalent. This
was obtained independently in
\cite[Theorem 1.2]{FSSC2} and  \cite[Theorem A.2]{garnie}. However, the
method goes back to the classical result by Friedrichs \cite{Fr}. The result can be stated as follows. 

\begin{theorem}\label{density}  Let $1\leq p <\infty$. We have
$$
\mathbf C^{\infty}(\Omega)\cap W^{1,p}_{\G}(\Omega) \mbox{ is dense in }
  W^{1,p}_{\G}(\Omega).
$$
\end{theorem}


   For functions in $W^{1,p}_\G(\Omega)$, the difference quotients are uniformly bounded by the $L^p$ norm of the weak gradient. This is expressed by the following Lemma, whose Euclidean counterpart is a classical result (see for instance \cite{Giusti2}).
   \begin{lemma}
    Let $f\in W^{1,p}_\G(\Omega)$. Then, for every $i=1,\dots, m_1$ and for every $\Omega'\subset\subset \Omega$ it holds 
    \[\|D_{i}^hf\|_{L^p(\Omega')}\leq \|X_i f \|_{L^p(\Omega)}\]for every $0<|h|<\dist(\Omega',\partial \Omega)$.
\end{lemma}
\begin{proof}
 By Theorem \ref{density}, it suffices to prove the statement for $f\in C^\infty(\Omega)\cap W^{1,p}_\G(\Omega)$.  
    Fix $i=1,\dots, m_1$, $\Omega'\subset\subset\Omega$ and $0<|h|<\dist(\Omega',\partial \Omega).$
    Consider $x\in \Omega'$ and define \[\gamma_x(s):=\exp(hsX_i)(x).\]
    Then $\gamma_x$ is a horizontal curve satisfying $\gamma_x(0)=x,\, \gamma_x(1)=\exp(hX_i)(x)$ and, by construction, $\dot\gamma_x(s)=hX_i(\gamma_x(s))$. Then the result follows by a direct computation:
     \begin{align*}\|D_{i}^hf\|^p_{L^p(\Omega')}&=\int_{\Omega'}\left|\frac{f(\exp(hX_i)(x))-f(x)}{h}\right|^p dx=\int_{\Omega'}\left|\frac{f(\gamma_x(1))-f(\gamma_x(0))}{h}\right|^p\,dx\\ &=\frac{1}{|h|^p}\int_{\Omega'} \left|\int_0^1\frac{d}{ds}(f\circ \gamma_x)(s)\,ds\right|^pdx=\frac{1}{|h|^p}\int_{\Omega'} \left|\int_0^1\big\langle \nabla_\G f(\gamma_x(s)),\dot\gamma_x(s)\big\rangle_{\gamma_x(s)}\,ds\right|^pdx\\
     &= \frac{1}{|h|^p}\int_{\Omega'} \left|\int_0^1hX_if(\gamma_x(s))\,ds\right|^pdx\leq\int_{\Omega'} \int_0^1\left|X_if(\gamma_x(s))\right|^pds\,dx\\ &=\int_0^1\int_{\Omega'} \left|X_if(\gamma_x(s))\right|^pdx\,ds =\int_0^1\int_{\Omega'\cdot \delta_{hs}(e_i)} \left|X_if(z)\right|^pdz\,ds\\
     &\leq \int_0^1\int_{\Omega} \left|X_if(z)\right|^pdz\,ds =\|X_if\|^p_{L^p(\Omega)}.\qedhere
     \end{align*}
\end{proof}
\begin{proposition}\label{Pweakdiff}    Let $1\leq p<+\infty$. Let $f:\,\Omega\subset\G\to\R$ and suppose that
\begin{itemize}
\item[(i)] $f$ is P-differentiable a.e. in $\Omega$;
\item[(ii)] $f\in W^{1,p}_{\G,\rm loc}(\Omega)$.
\end{itemize}
Then
\[
X_if(x)=\lim_{t\to 0}\frac{f\left(\exp(tX_i)(x)\right)-f(x)}{t}\text{ for a.e. }x\in\Omega\,
\]
for each $i=1,\dots,m_1$, where $X_if$ denotes the distributional derivative of $f$ along the vector field $X_i$.
\end{proposition}
\begin{proof}

Let us first show that, since $f\in W^{1,p}_{\G,\rm loc}(\Omega)$, then 
\[D_i^hf\to X_if \quad\text{in } L^p_{\rm loc}(\Omega).\]
Let $\Omega'\subset\subset \Omega$ and $u\in C^{\infty}(\Omega)\cap W^{1,p}_\G(\Omega)$. 
We write
\[D_i^h f-X_if=D_i^h(f-u)+D_i^hu-X_iu+X_i(u-f)\]
and using the previous lemma for $h$ sufficiently small
\[\|D_i^hf-X_if\|_{L^p (\Omega')}\leq \|D_i^hu-X_iu\|_{L^p(\Omega')}+2\|X_i(u-f)\|_{{L^p}(\Omega)}.\]
Now the claim follows by Theorem \ref{density}, since for $u\in C^{\infty}(\Omega)$ it is well known that $D_i^hu\to X_iu$ (which coincides with the classical directional derivative) \textit{uniformly} on compact sets. 

On the other hand, by assumption (i), $D_i^hf$ converges pointwise almost everywhere to the classical derivative of $f$ along the vector field $X_i$. Hence the conclusion follows.
\end{proof}
We conclude this section by recalling a well-known estimate on the Riesz potential that will be helpful later on. The proof easily follows using the Poincar\'e inequality in Carnot groups \cite {jerison}, \cite {FLW representation} and \cite {FW singapore} and we omit it.

\begin{theorem}\label{representation} Let $U=U(x,r)$ and let $f:\G\to \R$ be a
continuously differentiable function. Then there exists 
$C>0$, independent of $U$ and $f$, such that
\begin{equation}\label{riesz 1 formula}
|f(x)-f_{U}|\le C\int_{U}\frac{|\nabla_\G f(y)|}{\;\;\;\;d(x,y)^{Q-1}}\;dy 
\end{equation}
for all $x\in U$, where $f_{U}:=\ave_U f\,dx $. 
 Moreover, if $f\in W^{1,1}_{\G,\rm loc}(U)$, then \eqref{riesz 1 formula} holds for a.e. $x\in U$ and $\nabla_\G f(y)$ denotes the weak horizontal gradient of $f$ at $y$.
\end{theorem}

\subsubsection{Sobolev functions between Carnot groups.} We are now going to introduce the notions of $L^p$ functions and Sobolev functions with values in a Carnot group. These concepts actually work for mappings between general metric spaces and can be introduced in several different ways (see Section \ref{DiffMS}). A possible approach relies on \textit{upper gradients} (this is for instance the case discussed in \cite{LZ}), while here we use the definition by \textit{sections of mappings}. These notions go back to Ambrosio \cite{ambrosio2} and Reshetnyak \cite{Resh} for functions defined on open sets of $\Rn$ or a Riemannian manifold and have been extended by Vodop'yanov \cite{Vo99} to the setting of Carnot groups. A recent account containing  different approaches to this argument is due to Kleiner,  M\"uller and Xie \cite{KMX, KMX2}.
\begin{definition}\label{SobspacesvaluedCarnot} Let $\G_1=(\Rn,\cdot,d_1)$ and $\G_2=(\Rm,\star,d_2)$ be two Carnot  groups and let  $\Omega\subset \G_1$ be an open set.
\begin{itemize}
\item[(i)] We say that a map  $f:\,\Omega\to\G_2$ belongs to $L^p(\Omega,\G_2)$, and we will write $f\in L^{p}(\Omega,\G_2)$, if $f:\,\Omega\subset\Rn\to\Rm$ is measurable and the map $\Rn\ni x\mapsto d_2(0,f(x))$ is in $L^p(\Omega)$.
\item[(ii)]  We say that a map  $f:\,\Omega\to\G_2$ belongs to the Sobolev space $W^{1,p}(\Omega,\G_2)$, and we will write $f\in W^{1,p}(\Omega,\G_2)$, if 
\subitem ($\bullet$) $f\in L^{p}(\Omega,\G_2)$;
 \subitem ($\bullet$) for all $z\in\G_2$ the functions $u_z:\,\Omega\to\R$
\begin{equation}\label{defuz}
u_z(x):=\,d_2(z,f(x))\text{ if }x\in\Omega\,,
\end{equation}
are in $W^{1,p}_{\G_1,\rm loc}(\Omega)$;
\subitem ($\bullet$) there exists a nonnegative function $g\in L^p(\Omega)$ (independent of $z$) such that
\begin{equation}\label{estimnablauz}
|\nabla_{\G_1}u_z|\le\,g\text{ a.e. in $\Omega$, for each $z\in \G_2$}\,.
\end{equation}
\end{itemize}
In the following we will refer to a function $u_z$ as a section of $f$.
\end{definition}
As usual, we will denote by $W^{1,p}_{\rm loc}(\Omega,\G_2)$ the set of functions $f:\Omega\to\G_2$ such that $f\in W^{1,p}(\Omega',\G_2)$ for each $\Omega'\subset\subset\Omega$. 
\begin{remark}\label{rmk} The stability condition requested in \eqref{estimnablauz} with respect to the family of $1$-Lipschitz maps $\vf:\,\G_2\to\R$, $\vf(y):=\,d_2(z,y)$, can be extended to all Lipschitz maps between $\G_2$ and $\R$. Indeed the following characterization holds (see  \cite[Proposition 1]{Vo99} and \cite[Proposition 2.31]{KMX}, see also Lemma \ref{CRlipW}):  $f\in W^{1,p}(\Omega,\G_2)$ if and only if there exists a nonnegative function $g\in L^p(\Omega)$ such that,  for each Lipschitz function $\vf:\,\G_2\to\R$, the composition map $\vf\circ f:\Omega\to\R$ is in $W^{1,p}_{\G_1, \rm loc}(\Omega)$ and $|\nabla_{\G_1} (\vf\circ f)|\le\, {\rm Lip}(\vf)\,g$ a.e. in $\Omega$.
\end{remark}
Notice that, exploiting this remark, it is easy to see that Definition \ref{SobspacesvaluedCarnot} does not depend on the specific distance $d_2$ on $\G_2$ which appears in \eqref{defuz}. In fact, suppose $\tilde d_2$ is another invariant distance on $\G_2$. Then $d_2$ and $\tilde d_2$ determine the same Sobolev spaces. Indeed, since $d_2$ and $\tilde d_2$ are both invariant, they are equivalent as distances (see \cite[Corollary 5.1.5]{BLU}), namely there exists a constant $L>1$ such that
\begin{equation}\label{eq}L^{-1}d_2(x,y)\leq \tilde d_2(x,y)\leq L d_2(x,y)\end{equation}
for every $x,y\in\G_2$. 
Suppose now $f\in W^{1,p}(\Omega,\G_2)$ according to the distance $d_2$. Then
\begin{itemize}
\item The map $\Rn\ni x\mapsto \tilde d_2(0,f(x))$ is in $L^p(\Omega)$ by \eqref{eq};
\item For every $z\in\G_2$, the map $\varphi_z:(\G_2, d_2)\to(\R,|\cdot|)$ defined by $\varphi_z(y):=\tilde d_2(z,y)$ is $L$-Lipschitz since
\[|\varphi_z(x)-\varphi_z(y)|=|\tilde d_2(z,x) -\tilde d_2(z,y)|\leq \tilde d_2(x,y)\leq Ld_2(x,y).\]
Hence we can apply Remark \ref{rmk} and deduce that 
the composition $\tilde d_2(z,f(\cdot))=\varphi_z\circ f\in W^{1,p}_{\G_1,\rm loc}(\Omega)$;
\item Again by the previous considerations and Remark \ref{rmk}, we conclude that $|\nabla_{\G_1}(\varphi_z\circ f)|\leq Lg$ a.e. in $\Omega$ for each $z\in\G_2$.
\end{itemize}Thus $f\in W^{1,p}(\Omega,\G_2)$ with respect to $\tilde d_2$: this shows that Definition \ref{SobspacesvaluedCarnot} does not depend on the specific invariant distance on $\G_2$.\\\\
P-differentiability and  Morrey embedding theorem  for real-valued Sobolev spaces on  Carnot groups when $p>Q$ (see, for instance, \cite{MSC} and \cite[Theorem 1.11]{garnie}) can be extended to Sobolev maps between Carnot groups. Here $Q$ always denotes the homogeneous dimension of $\G_1$.
\begin{theorem}[{\rm P-differentiability with $p>Q$, \cite[Corollary 2]{Vo03}}] \label{pdifferent}Let $f\in W^{1,p}(\Omega,\G_2)$ with $p>\,Q$. Then $f$ is P-differentiable a.e. in $\Omega$.
\end{theorem}
The following theorem can be found in \cite[Lemma 2.7]{KMX2}.
\begin{theorem}[{\rm Morrey-Sobolev embedding with $p>Q$}]\label{Sobembp>Q}
Let $f: U(x,r)\subset\G_1\to\G_2$ and $p>\,Q$. There is a positive constant $C_p$ (depending on $p$) such that,  if $f\in W^{1,p}(U(x,r),\G_2)$ and $g$ is as in \eqref{estimnablauz}, then $f$ has a continuous representative and for each $y\in U(x,r)$ it satisfies
\begin{equation}\label{embedp>Q}
d_2(f(x),f(y))\le\, C_p\,r^{1-\frac{Q}{p}}\,\left(\int_{U(x,r)}g^p\,d\mathcal L^n\right)^{1/p}.
\end{equation}
\end{theorem}


\section{Absolutely continuous functions between Carnot groups and some properties}\label{ACF}
Let us collect in this section some interesting properties of $Q-$absolutely continuous functions and functions of bounded $Q-$variation.\\
Let us begin to observe that, on bounded open sets (with finite measure would suffice), the class of $Q$-absolutely functions is contained in the one of bounded $Q$-variation functions.
\begin{proposition}\label{ACsubsetBV} Let $\Omega\subset\G_1$ be a bounded  open set. Then $AC^Q(\Omega,\G_2)\subset BV^Q(\Omega,\G_2)$.
\end{proposition}
\begin{proof} The strategy of the proof is inspired by \cite[Lemma 2.16]{LZ}, where the authors prove a local version of this result for abstract metric measure spaces. 
Let's denote by $A:=\leb{n}(\Omega)<+\infty$ and let $f\in AC^Q(\Omega,\G_2)$. Fix $\epsilon=1$: by $Q-$absolute continuity, there exists $\delta >0$ such that, for each finite disjoint family of open balls $U_1,\dots,U_k$ contained in $\Omega$,
\begin{equation}\label{abscon}
\text{if }\sum_{i=1}^k\leb n (U_i)<2\delta\,\text{ then }\,\sum_{i=1}^k\left({\rm osc}_{U_i}f\right)^Q<\,1.\end{equation}
Let $n\in\N$ be sufficiently large so that $A<n\delta$.
Let now $\mathcal F=\{U_1,\dots, U_k\}$ be a disjoint family of balls in $\Omega$ 
and write $\mathcal F=\mathcal F_1\cup \mathcal F_2$, where 
\[\mathcal F_1=\{U_j\in\mathcal F: \leb{n}(U_j)\geq \delta\},\]
\[\mathcal F_2=\{U_j\in\mathcal F: \leb{n}(U_j)< \delta\}.\]
Notice that the elements in $\mathcal F_1$ are at most $n$ 
and their radii are equibounded. 
For any $U=U(x,r)\in\mathcal F_1$ and any $y\in U$ consider a geodesic joining $x$ and $y$. On such a curve pick points $x=x_0,\,x_1,\dots,\,x_{N-1},\,x_N=y$ such that $x_i$ and $x_{i+1}$ live in a common small ball (contained in $U$) of measure less than $2\delta$. 
Note that it is possible to find $N$ independent of $y$ and $U$ since all the elements in $\mathcal F_1$ have equibounded radii. By \eqref{abscon} we can now estimate
\[\begin{split}d_2(f(x),f(y))^Q&\leq (d_2(f(x),f(x_1))+\dots+d_2(f(x_{N-1}),f(y)))^Q\\ &\leq C(Q)(d_2(f(x),f(x_1))^Q+\dots+d_2(f(x_{N-1}),f(y))^Q)\leq C(Q)N:=K.\end{split}\]
By the triangular inequality the same estimate holds (with a greater constant $\widetilde K$) for any $y,z\in U$.
We conclude that $\sum_{U\in \mathcal F_1}(\text{osc}_{U}f)^Q\leq n\widetilde K.$\\
 Collect instead the elements of $\mathcal F_2$ into at most $n$ sets $\mathcal F_{2,1},\dots\mathcal F_{2,n}$ with $\sum_{U\in\mathcal F_{2,j}}\leb{n}(U)<2\delta$, so that by \eqref{abscon}\[\sum_{U\in \mathcal F_2}(\text{osc}_{U}f)^Q=\sum_{j=1}^{n}\sum_{U\in\mathcal F_{2,j}}(\text{osc}_{U}f)^Q< n\]
 Hence $V_Q(f,\Omega)<n(\widetilde K+1)$ and the thesis follows.
\end{proof}
It is clear that a $Q$-absolutely continuous function is actually continuous; here we provide a first example of a class of $Q$-absolutely continuous maps.
\begin{remark}\label{lipimpliesabs} Let $f:\,\Omega\subset \G_1\to\G_2$ be Lipschitz continuous. Then $f$ turns out to be $Q$-absolutely continuous.  Indeed, 
\begin{equation*}
{\rm osc}_{U_r(x)}(f):=\diam_{\G_2}(f(U_r(x)))\le \mathrm{Lip}(f)\,\diam_{\G_1}(U_r(x))
\end{equation*}
for each $x\in\G_1$ and $r>\,0$. Using the previous inequality we immediately get that there exists $L_1>\,0$ such that 
\begin{equation}\label{contrlipballs}
{\rm osc}_{U}(f)^Q\le (2\,\mathrm{Lip}(f))^Q\,r^Q= L_1\,|U|
\end{equation}
for each (open) CC  open ball $U\subset\G_1$ and the conclusion follows.
\end{remark} 

\begin{definition}\label{RRdef} A function $f:\,\Omega\subset\G_1\to\G_2$ satisfies the  condition (RR) if there is a nonnegative function $w\in L^1(\Omega)$, called weight, such that
\begin{equation}\label{RRcond}
{\rm osc}_{U_r(x)}(f)^Q\le\,\int_{U_r(x)}w\,dy
\end{equation}
for every ball $U_r(x)\subset\Omega$, being $Q$ the homogeneous dimension of $\G_1$.
\end{definition}

\begin{remark} The acronym (RR) stands for Rado and Reichelderfer \cite{RR}, who introduced a similar condition for proving an area formula in $\mathbb{R}^n$. The condition (RR), as stated in Definition \ref{RRdef}, is a generalization to Carnot groups of the one due to Mal\'y \cite{Maly}.
\end{remark}
\begin{remark} It is immediate to check that if $f:\Omega\to\G_2$ satisfies the  condition (RR), then $f\in AC^Q(\Omega,\G_2)$, by the absolute continuity of the Lebesgue integral.  The viceversa is known in Euclidean spaces for functions $f\in AC^n(\Omega,\Rm)$ \cite[Thm. 9]{Cso} and it is still open for maps between Carnot groups. 

\end{remark}

As observed in Remark \ref{lipimpliesabs}, Lipschitz functions are $Q-$absolutely continuous. Also Sobolev functions with a high order of summability share the same property. This result was shown in \cite[Thm. 4.1]{Maly} for the Euclidean case and the proof can be carried out in the same way by exploiting Morrey-Sobolev embedding Theorem \ref{Sobembp>Q}. Here we mention an alternative approach which relies on an equivalent definition of Sobolev spaces through upper gradients.
\begin{theorem}\label{p>Qdiff}
 Let $\Omega\subset\G_1$ be an open bounded set 
 and let $f\in W^{1,p}(\Omega,\G_2)$ with $p> Q$, where $Q$ is the homogeneous dimension of  $\G_1$. Then, for a suitable representative, $f\in AC^Q(\Omega,\G_2)$.
\end{theorem} 
\begin{proof}
By the following Corollary \ref{metricvaluedNG=WG}, we can assume $f\in N^{1,p}(\Omega:\G_2)$, where $N^{1,p}(\Omega:\G_2)$ denotes the Newtonian Sobolev space between $\Omega$ and $\G_2$ (see Appendix \ref{DiffMS}). Applying \cite[Theorem 1.4]{LZ} we immediately get the conclusion.
\end{proof}

We are now interested in studying some differentiability properties of absolutely continuous maps between Carnot groups. We start with the following 
\begin{theorem}[Differentiability a.e.]\label{diffthm} Let $f\in BV^Q(\Omega,\G_2)$. Then 
\begin{itemize}
\item[(i)] $f$ is P-differentiable for a.e. $p\in\Omega$;
\item[(ii)] ${\rm Lip}(f)\in L^Q(\Omega)$, where ${\rm Lip}(f)$ is the function defined in \eqref{puLcf}.
\end{itemize}
\end{theorem}

\begin{remark} Note that  Pansu's differentiability for maps in $BV^Q(\Omega,\G_2)$ does not follow from  differentiability  result à la Cheeger-Kleiner proved in \cite[Proposition 1.3]{LZ} for maps between metric spaces, of which the target metric space is biLipschitz embedded into a Banach space satisfying the Radon- Nikodym property (RNP).  Indeed it is well-known that Carnot groups may not be biLipschitz embedded into a Banach space satisfying (RNP) (see Appendix \ref{DiffMSemb})
\end{remark}

\begin{proof} We can follow the same proof of \cite[Thm. 3.3]{Maly} with suitable changes. For the sake of completeness, we collect here the main steps. Let $g:\Omega\to \R$ be a nonnegative upper semicontinuous function such that
    \[0\leq g(x)\leq [Lip(f)(x)]^Q\qquad\text{for all }x\in\Omega.\]
    For any $x\in\{g>0\}$, there exists a unique integer $k=k(x)$ such that $2^k\leq g(x)<2^{k+1}$. Furthermore, since $g$ is upper semicontinuous, there exists $0<r=r(x)<1$ such that $U(x,5r)\subset \{g<2^{k+1}\}$. In addition, from the fact that $g(x)\leq [Lip(f)(x)]^Q$, we can assume, up to reduce $r$, that 
    \begin{equation}\label{esti}2^kr^Q\leq g(x)r^Q<2(\underset{U(x,r)}{\text{osc}} f)^Q.\end{equation}
    Applying Vitali's covering Lemma, we can extract a disjoint (hence at most countable) family $\{U(x_i,r_i)\}_i\subset \{U(x,r(x))\}$ of balls such that 
    \[\{g>0\}\subset\bigcup_{x\in\{g>0\}}U(x,r(x))\subset\bigcup_iU(x_i,5r_i).\]
    We define, for all $k\in\mathbb Z,$ $I_k:=\{i:k(x_i)\geq k\}$ and $E_k:=\{x\in\Omega:2^k\leq g(x)< 2^{k+1}\}$.\\
    Notice that $E_k\subset \bigcup_{i\in I_k}U(x_i,5r_i)$. In fact, let $x\in E_k$. In particular $g(x)>0$, hence $x\in U(x_i,5r_i)$ for some $i$. If it were $i\not\in I_k$, then $k(x_i)<k$, which implies 
    $x\in U(x_i, 5r_i)\subset\{g<2^k\}$, contradicting $x\in E_k$. Thus we get
    \[\begin{split}\mathcal L^n(E_k)&\leq \sum_{i\in I_k}\mathcal L^n(U(x_i,5r_i))\leq C\sum_{i\in I_k}r_i^Q\\ &\stackrel{(\ref{esti})}\leq C\sum_{i\in I_k}2^{-k(x_i)}(\underset{U(x_i,r_i)}{\text{osc}} f)^Q.
    \end{split}\]
    Finally we use this estimate to bound
    \[\begin{split}\int_\Omega g(x)\, dx&\leq \sum_{k=-\infty}^{+\infty}\int_{E_k} g(x)\, dx\leq \sum_{k=-\infty}^{+\infty}2^{k+1}\mathcal L^n(E_k) \\ &\leq C\sum_{k=-\infty}^{+\infty} \sum_{i\in I_k}2^{k+1-k(x_i)}(\underset{U(x_i,r_i)}{\text{osc}} f)^Q\\
    &=C\sum_{i=1}^{+\infty}\sum_{k\leq k(x_i)}2^{k+1-k(x_i)}(\underset{U(x_i,r_i)}{\text{osc}}f)^Q\\ &=4C\sum_{i=1}^{+\infty}(\underset{U(x_i,r_i)}{\text{osc}}f)^Q\leq CV_Q(f,\Omega),\end{split}\]
    where we exploited the properties of the $Q-$variation and the fact that $\sum_{k\leq k(x_i)}2^{k}= 2^{k(x_i)+1}$. 
It follows that\begin{small}
\[\int_\Omega [Lip(f)(x)]^Q\,dx=\sup\Big\{\int_\Omega g(x)\,dx: g \text{ u.s.c., } 0\leq g\leq Lip(f)^Q\Big\}\leq CV_Q(f,\Omega).\]
\end{small}
Thus  (ii) follows and also that ${\rm Lip}(f)(x)<\infty$ for a.e. $x\in\Omega$. From Theorem \ref{Stepanov}, (i) follows, too.
\end{proof}

We will now focus our attention to the case when $\G_2=\,(\Rm,+)$. Note that in this case, since $\Rm$ is abelian, a map  $f=\,(f_1,\dots,f_m):\,\Omega\subset\G\to\Rm$  is $Q$-absolutely continuous (respectively of finite $Q$-variation) if and only if each component $f_i:\,\Omega\subset\G\to\R$ is $Q$-absolutely continuous (respectively of finite $Q$-variation), for each $i=1,\dots, m$. Thus we can reduce the study to real-valued  functions defined on an open set of a Carnot group. We will denote
\[
BV_\G^Q(\Omega):=\, BV^Q(\Omega,\R),\,AC_\G^Q(\Omega):=\,AC^Q(\Omega,\R)\,,
\]
\[
BV_{\G;\rm loc}^Q(\Omega):=\, BV_{\rm loc}^Q(\Omega,\R),\,AC_{\G;\rm loc}^Q(\Omega):=\,AC_{\rm loc}^Q(\Omega,\R)\,.
\]

The following theorem can be proved mutatis mutandis as its Euclidean counterpart, and the proof relies on standard arguments. 
\begin{theorem}\label{BVACBspace}
If $f\in L^Q(\Omega)\cap BV_\G^Q(\Omega)$ let us denote
\[
\|f\|_{Q,BV}:=\,\|f\|_{L^Q(\Omega)}+V_Q(f,\Omega)^{1/Q}.
\]
Then $(L^Q(\Omega)\cap BV_\G^Q(\Omega), \|\cdot\|_{Q,BV})$ is a Banach space. Moreover  $L^Q(\Omega)\cap AC_\G^Q(\Omega)$ is a closed linear subspace of $L^Q(\Omega)\cap BV_\G^Q(\Omega)$.

\end{theorem}

\begin{proof} 

It is easy to see that $\|\cdot\|_{Q,BV}$ is a norm on $L^Q(\Omega)\cap BV_\G^Q(\Omega)$. Let us prove that it is a Banach norm.  Let $(f_h)_h \subset  (L^Q(\Omega)\cap BV_\G^Q(\Omega), \|\cdot\|_{Q,BV})$ be a Cauchy sequence. It is sufficient to prove that there exist $f\in L^Q(\Omega)\cap BV_\G^Q(\Omega)$ and a subsequence $(f_{h_k})_k$ satisfying 
\begin{equation}\label{BVACBspace1}
f_{h_k}\to f\text{  in }L^Q(\Omega)\text{, as }k\to\infty\,,
\end{equation}
and
\begin{equation}\label{BVACBspace2}
\lim_{k\to\infty}V_Q(f_{h_k}-f,\Omega)=\,0.
\end{equation}
Since $ (f_h)_h \subset  (L^Q(\Omega)\, \|\cdot\|_{L^Q(\Omega)})$ is also a Cauchy sequence, we can suppose the
re exists a function $f^*\in L^Q(\Omega)$ such that  \eqref{BVACBspace1} holds with $f\equiv f^*$ and $h\equiv h_k$. Let us now prove that, up to a subsequence,
\begin{equation}\label{BVACBspace3}
(f_h(x))_h\text{ is a Cauchy sequence in }\R \quad\text{for every }x\in\Omega.
\end{equation}
Indeed, we can assume that, up to a subsequence,  there exists a $\L{n}$-negligible set $N\subset\Omega$ such that
\begin{equation}\label{BVACBspace4}
f_h(x)\to f^*(x)\in\R \text{, as  }h\to\infty,\text{ for each }x\in\Omega\setminus N\,.
\end{equation} 
Since $(f_h)_h$ is a Cauchy sequence with respect to  the norm $\|\cdot\|_{Q,BV}$, we can assume that, for each $\eps>\,0$ there exists $\bar h=\,\bar h(\eps)\in\N$ such that
\begin{equation}\label{BVACBspace5}
\begin{split}
&\sum_{i=1}^m |f_h(x^{(i)})-f_k(x^{(i)})-f_h(y^{(i)})+ f_k(y^{(i)})|^Q\\
&\le \sum_{i=1}^m{\rm osc}_{U_i}(f_h-f_k)^Q\\
&\le\,V_Q(f_h-f_k,\Omega)<\,\eps,
\end{split}
\end{equation}
for each $h,\, k>\,\bar h$,   for any finite disjoint family of balls $U_1,\dots,U_m$ in $\Omega$, for each $x^{(i)},\,y^{(i)}\in U_i$ and $i=1,\dots,m$. 

Fix now a ball $U(x,r)\subset\Omega$. By the density of $\Omega\setminus N$ in $\Omega$, there exists $x_0\in U(x,r)\cap (\Omega\setminus N)$. By \eqref{BVACBspace5}, the sequence
\[
g_h(x):=\,f_h(x)-f_h(x_0)\quad h\in\N\,,
\]
is a Cauchy sequence. On the other hand, by \eqref{BVACBspace4},  the sequence $(f_h(x_0))_h$ is a Cauchy sequence. Thus \eqref{BVACBspace3} follows. From \eqref{BVACBspace3}  , it follows that there exists a function $f:\Omega\to\R$ defined as
\[
f(x):=\,\lim_{h\to\infty}f_h(x)\quad\text{ for each }x\in\Omega\,,
\]
and, by  \eqref{BVACBspace4},
\[
f(x)=\,f^*(x)\quad\text{for each }x\in\Omega\setminus N\,.
\]
In particular, \eqref{BVACBspace1} immediately follows.   Let us now prove that  $f\in BV_\G^Q(\Omega)$, that is,
\begin{equation}\label{BVACBspace6}
V_Q(f,\Omega)<\infty\,.
\end{equation}
By passing to the limit in \eqref{BVACBspace5} as $k\to\infty$, we get,  for each $\eps>\,0$ there exists $\bar h=\,\bar h(\eps)\in\N$ such that
\begin{equation*}
\,\sum_{i=1}^m |f_h(x^{(i)})-f(x^{(i)})-f_h(y^{(i)})+ f(y^{(i)})|^Q\le\,\eps,
\end{equation*}
for each $h>\,\bar h$,   for any finite disjoint family of balls $U_1,\dots,U_m$ in $\Omega$, for each $x^{(i)},\,y^{(i)}\in U_i$ and $i=1,\dots,m$.  Taking the supremum on each ball $U_i$, and then the supremum on all finite, disjoint family of balls in $\Omega$, it follows that,  for each $\eps>\,0$ there exists $\bar h=\,\bar h(\eps)\in\N$ such that
\begin{equation}\label{BVACBspace7}
V_Q(f_h-f,\Omega)\le\, \eps\text{ for each }h>\,\bar h\,.
\end{equation}
Observe now that $f\mapsto V_A(f,\Omega)^{1/Q}$   satisfies the triangle inequality. Thus, from \eqref{BVACBspace7}, if $h>\,\bar h$
\begin{equation*}
\begin{split}
V_A(f,\Omega)^{1/Q}&\le\,V_A(f-f_h,\Omega)^{1/Q}+V_A(f_h,\Omega)^{1/Q}\\
& \le\,\eps^{1/Q}+V_A(f_h,\Omega)^{1/Q}\,,
\end{split}
\end{equation*}
and \eqref{BVACBspace6} follows since $V_A(f_h,\Omega)^{1/Q}<\,\infty$. Applying again \eqref{BVACBspace7}, \eqref{BVACBspace2} follows, too.

Finally let us show that $L^Q(\Omega)\,\cap\, AC_\G^Q(\Omega)$ is a closed linear subspace of $(L^Q(\Omega)\cap BV_\G^Q(\Omega), \|\cdot\|_{Q,BV})$. Let $(f_h)_h \subset L^Q(\Omega)\cap AC_\G^Q(\Omega)$ and assume there exists $f\in L^Q(\Omega)\cap BV_\G^Q(\Omega)$ and $\lim_{h\to\infty}\|f_h-f\|_{Q,BV}=\,0$. In particular note \eqref{BVACBspace7} holds. Now observe that, by \eqref{BVACBspace7}, if $h>\,\bar h$,  given a finite, disjoint family of balls $U_1,\dots,U_m$ in $\Omega$, 
\begin{equation}\label{BVACBspace8}
\begin{split}
\left(\sum_{i=1}^m{\rm osc}_{U_i}(f)^Q\right)^{1/Q}&\le\,\left(\sum_{i=1}^m{\rm osc}_{U_i}(f-f_h)^Q\right)^{1/Q}
+\left(\sum_{i=1}^m{\rm osc}_{U_i}(f_h)^Q\right)^{1/Q}\\
&\le\,V_Q(f-f_h)^{1/Q}+\left(\sum_{i=1}^m{\rm osc}_{U_i}(f_h)^Q\right)^{1/Q}\\
&\le\, \eps^{1/Q}+\left(\sum_{i=1}^m{\rm osc}_{U_i}(f_h)^Q\right)^{1/Q}\,.
\end{split}
\end{equation}
 Since each $f_h\in AC_\G^Q(\Omega)$ for each $h$, by \eqref{BVACBspace8},  it follows that $f\in AC_\G^Q(\Omega)$, too.
\end{proof}
Theorem \ref{diffthm} states that $BV^Q$-functions are P-differentiable pointwise almost everywhere. We will now prove that such maps enjoy additional differentiability properties, namely they are weakly differentiable according to Definition \ref{SobspacesvaluedCarnot}.
More precisely, 
\begin{theorem}[Weak differentiability]\label{Weakdiff} If $Q>\,1$, then \[BV^Q_{\rm loc}(\Omega,\G_2)\subset W^{1,Q}_{\rm loc}(\Omega,\G_2).\]
\end{theorem}
\begin{remark} Note that  from \cite[Theorem 1.2]{LZ} it follows that $BV^Q_{\rm loc}(\Omega,\G_2)\subset D^{1,Q}_{\rm loc}(\Omega:\G_2)$, where $D^{1,Q}_{\rm loc}(\Omega:\G_2)$ denotes the (local) {\it Dirichlet class} of maps between metric spaces $(\Omega,d_1,\mathcal L^n)$ and $(\G_2,d_2)$ according to  \cite[Chap. 7]{HKST} (see Appendix \ref{MSvaluedSM}). It is also straightforward from definition that $BV^Q_{\rm loc}(\Omega,\G_2)\subset L^\infty_{\rm loc}(\Omega,\G_2)$. By definition,  $D^{1,Q}_{\rm loc}(\Omega:\G_2)\cap L^Q_{\rm loc}(\Omega,\G_2)=\,N^{1,Q}_{\rm loc}(\Omega:\G_2)$, where $N^{1,Q}_{\rm loc}(\Omega:\G_2)$ denotes the (local) {\it Newtonian Sobolev space} of maps between the previous metric spaces. 
On the other hand, from Corollary \ref{metricvaluedNG=WG}, it follows that $N^{1,Q}_{\rm loc}(\Omega:\G_2)=\,W^{1,Q}_{\rm loc}(\Omega,\G_2)$.  Therefore  Theorem \ref{Weakdiff} can be shown also by this strategy.
\end{remark}

Here we propose an alternative proof of Theorem \ref{Weakdiff} with respect to the one in \cite[Theorem 1.2]{LZ}, by using the intrinsic convolution in Carnot groups and in the same spirit of \cite[Theorem 3.2]{Maly}. In order to prove Theorem \ref{Weakdiff} we need the following  technical lemma. Point (i) establishes a uniform estimate on the distance between the images of integer points after a right-translation. Point (ii) provides a way to cover $\G$ with arbitrarily small balls centered at suitable dilations of the integer lattice.
\begin{lemma}\label{estimdistintegers}
\begin{itemize}
\item[(i)] For every $R_0>0$, it holds that
\begin{equation}\label{sigma0} 
\sigma_0:=\inf\{d_c\left(\xi\cdot p,\eta\cdot p\right):\,\xi,\eta\in\Z^n,\,\xi\neq \eta,\, p\in B(0,R_0)\}>\,0\,.
\end{equation}
\item[(ii)] Given a positive integer $k$, let ${\bf Z}_k$ be  the set of all points $z\in\G\equiv\Rn$ such that $\delta_k(z)$ has integer coordinates, that is,
\begin{equation}\label{Zk}
{\bf Z}_k:=\delta_{1/k}(\Z^n)\,.
\end{equation}
Then there exists a positive constant $\sigma_1>0$ such that
\[
\cup_{z\in {\bf Z}_k}U(z,\sigma_1/k)=\,\G\,.
\]
\end{itemize}
\end{lemma} 
\begin{proof}
(i) Since $d_c$ and $d_\infty$ are equivalent distances, it is sufficient to show that
\begin{equation}\label{sigmastar0} 
\sigma^*_0:=\inf\{d_\infty\left(\xi\cdot p,\eta\cdot p\right):\,\xi,\eta\in\Z^n,\,\xi\neq \eta,\, p\in B_\infty (0,R_0)\}>\,0\,.
\end{equation}

By contradiction, if $\sigma^*_0=\,0$, there exist three sequences $(\xi_h)_h,\,(\eta_h)_h$ in $ \Z^n$, $(p_h)_h$ in $B_\infty (0,R_0)$,with  $\xi_h\neq \eta_h$ and $\lim_{h\to\infty} d_\infty\left(\xi_h\cdot p_h,\eta_h\cdot p_h\right)=\,\lim_{h\to\infty}\|p_h^{-1}\cdot\eta_h^{-1}\cdot\xi_h\cdot p_h\|_\infty=\,0$. 
Let $x_h:=\eta_h^{-1}\cdot \xi_h$. We observe that $(x_h)_h$ is bounded in $(\R^n,d_\infty)$. In fact, for all $h$,
\[\begin{split}\|x_h\|_\infty&=\|p_h\cdot p_h^{-1}\cdot x_h\cdot p_h\cdot p_h^{-1}\|_\infty\\ &\leq \|p_h\|_\infty+\|p_h^{-1}\cdot x_h\cdot p_h\|_\infty+\|p_h^{-1}\|_\infty\\ &\leq 2R_0+\|p_h^{-1}\cdot x_h\cdot p_h\|_\infty.\end{split}\]
Hence, up to a subsequence, we can assume that $p_h\to p_*$ and $x_h\to x_*$ as $h\to\infty$. This implies, by the continuity of the norms and of the group operation, that
\begin{equation}\label{passoallimite}0=\lim_{h\to\infty}\|p_h^{-1}\cdot x_h\cdot p_h\|_\infty=\|p_*^{-1}\cdot x_*\cdot p_*\|_\infty,\end{equation}
\begin{equation}\label{passoallimite1}\|x_*\|_{\R^n}=\lim_{h\to\infty}\|\eta_h^{-1}\cdot \xi_h\|_{ \R^n}.\end{equation}
From (\ref{passoallimite}) it follows that $x_*=0$.\\
 Let us now show that
 \begin{equation}\label{integerdiffereneucl}
 \|\eta^{-1}\cdot\xi\|_{\Rn}\ge\,1\text{ for each }\xi,\, \eta\in\Z^n\text{, with }\xi\neq\eta\,,
\end{equation}
thus a contradiction will follow from \eqref{passoallimite1}. Let us represent now $\xi,\,\eta\in\Z^n$ by means of the layer representation 
\[
\xi=[\xi^1,\xi^2,\dots,\xi^s]\text{ and }\eta=[\eta^1,\eta^2,\dots,\eta^s]\,,
\]
with $\xi^i,\,\eta^i\in\Z^{m_i}$ ($i=1,\dots,s$) and, since $\xi\neq\eta$, we can  assume that $\xi^{i_0}\neq\eta^{i_0}$ for a suitable $i_0$. Here we recall that $\mathfrak g=V_1\oplus\dots\oplus V_s$ and $m_i=\dim(V_i)$.
Meanwhile, we can represent the element $\eta^{-1}\cdot\xi$ as
\[
\eta^{-1}\cdot\xi=[(\eta^{-1}\cdot\xi)^1,(\eta^{-1}\cdot\xi)^2,\dots, (\eta^{-1}\cdot\xi)^s]\,,
\]

\[
(\eta^{-1}\cdot\xi)^1=\xi^1-\eta^1,\,(\eta^{-1}\cdot\xi)^i=\xi^i-\eta^i+\mathcal Q^i(-\eta^1,\dots, -\eta^{i-1}, \xi^1,\dots,\xi^{i-1})
\]
with $i=2,\dots,s$ and $\mathcal Q^i:=(\mathcal Q_{h_{i-1}+1},\dots, \mathcal Q_{h_i})$. 
It is clear that if $\xi^1-\eta^1\neq 0$, then  $\|\eta^{-1}\cdot\xi\|_{\Rn}\ge\,\|\xi^1-\eta^1\|_{\R^{m_1}}\ge\,1$  and \eqref{integerdiffereneucl} holds.  If $\xi^1-\eta^1=\, 0$, then, by the properties of $\mathcal Q_i$' s (see  \eqref{legge di gruppo3}) $\mathcal Q^2(-\eta^1,\xi^1)=\,\mathcal Q^2(-\xi^1,\xi^1)=\,0$, which implies   $\|\eta^{-1}\cdot\xi\|_{\Rn}\ge\,\|(\eta^{-1}\cdot\xi)^2\|_{\R^{m_2}}=\,\|\xi^2-\eta^2\|_{\R^{m_2}}$. Thus if $\xi^2-\eta^2\neq\,0$, \eqref{integerdiffereneucl}  holds. Otherwise, in the baddest case, it could happen that $\xi^1-\eta^1=\,\dots=\,\xi^{i_0-1}-\eta^{i_0-1}=\,0$. This implies that 
$\mathcal Q^{i_0}(-\eta^1,\dots,-\eta^{i_0-1},\xi^1,\dots,\xi^{i_0-1})=\,\mathcal Q^{i_0}(-\xi^1,\dots,-\xi^{i_0-1},\xi^1,\dots,\xi^{i_0-1})=\,0$. Thus $\|\eta^{-1}\cdot\xi\|_{\Rn}\ge\,\|(\eta^{-1}\cdot\xi)^{i_0}\|_{\R^{m_{i_0}}}=\,\|\xi^{i_0}-\eta^{i_0}\|_{\R^{m_{i_0}}}\ge\,1$ and \eqref{integerdiffereneucl} holds.

\noindent(ii) It is well-known that  it is possible to construct a pavage in a Carnot group (see, for instance, \cite[Sect. 2.3]{FGuVG}). More precisely, if $Y:=\{x=(x_1,\dots,x_n):\, -1/2<\,x_i\le 1/2\}$ denote the unit cube in $\Rn$ centered at the origin, denote
\[
Q^{\eps}_\xi:=\,\delta_\eps(\tau_\xi(Y))\text{ if }\eps>\,0,\,\xi\in\Z^n\,.
\]
Then, it holds that, for each $\eps>\,0$, the family of sets $\{Q^{\eps}_\xi:\,\xi\in\Z^n\}$ is a {\it pavage} of  $\G=(\Rn,\cdot)$, that is,  the sets  are disjoint and
\begin{equation}\label{pavage}
\Rn=\,\cup_{\xi\in\Z^n} Q^{\eps}_\xi\,.
\end{equation}
It is easy to see that there exist $\sigma^*_1, \sigma_1>\,0$ such that
\begin{equation}\label{cubeball}
Y\subset U_\infty(0,\sigma^*_1)\subset U(0,\sigma_1)\,.
\end{equation}
Thus, choosing $\eps=1/k$, from \eqref{pavage} and \eqref{cubeball}, we get the desired conclusion.
\end{proof}
\begin{proof}[Proof of Theorem \ref{Weakdiff}] {\bf 1st step.}  Let us first show our conclusion in the case $\G_1\equiv\G$ and $\G_2=\,\R$.
Proof's strategy will be the same as \cite[Theorem 3.2]{Maly}. We will point out in the  following the main changes and differences which need in the setting of Carnot groups. First assume that $f\in BV^Q_\G(\Rn)\cap L^Q(\Rn)$. Let $\sigma_1$ be as in (ii) of the previous lemma and find $\sigma_0$ as in (i) for $R_0=2\sigma_1$. This means that, for each $y\in B(0,2\sigma_1)$,
\[\sigma_0\leq\inf\{d_{c}(\xi\cdot y,\eta\cdot y):\xi,\eta\in\mathbb Z^n,\xi\neq \eta\}.\]
This also implies that, for every integer $k$ and every $y\in B(0,2\sigma_1/k)$,
\[\frac{\sigma_0}{k}\leq\inf\{d_{c}(p\cdot y,q\cdot y):p,q\in\mathbf Z_k, p\neq q\}.\]
Hence, for every $y\in B(0,2\sigma_1/k)$, the family of CC balls $\{U(z\cdot y, \sigma_0/3k):z\in \mathbf Z_k\}$ is disjoint. In addition, up to increase $\sigma_1$, we can assume $\sigma_1\geq \sigma_0$. \\Let $(\rho_k)_k$ be the sequence of mollifiers defined in \eqref{mollifier}, where $\eps=1/k$, $\rho_k\equiv\rho_{1/k}$  and suppose that $\spt (\rho)\subset B(0,\sigma_0/6)$. Observe that it holds that
\begin{equation}\label{nablaGrhok}
\nabla_\G\rho_k(x)=\,k^{Q+1}\,\nabla_\G\rho(\delta_k(x))\text{ for each }x\in\Rn,\,k\in\N\,,
\end{equation}
\begin{equation}\label{intnablaGrhok}
\int_{\Rn}X_i\rho_k(r_{x}(y))\,dy=\,0\text{ for each }x\in\Rn,\,i=1,\dots,m_1\,,
\end{equation}
\begin{equation}\label{nablaGfstarrhok}
\begin{split}
\nabla_\G(g\star \rho_k)(x)&=\,\int_{\Rn}\nabla_\G\rho_k(r_{x}(y^{-1}))\,g(y)\,dy\\
&=\,\int_{\Rn}\nabla_\G\rho_k(w^{-1})\,g(\ell_x(w))\,dw\\
&\text{ for each }g\in L^1_{\rm loc}(\Rn),\,x\in\Rn\,.
\end{split}
\end{equation}
Indeed \eqref{nablaGrhok} follows by the definition of $\rho_k$ and homogeneity of intrinsic gradient $\nabla_\G:=(X_1,\dots,X_{m_1})$;\eqref{intnablaGrhok} follows by the properties of  $\rho_k$ , the structure of the horizontal vector fields $X_i$ and  the right-invariance of $n$-dimensional Lebesgue measure $\L n$;\eqref{nablaGfstarrhok} follows by the  definition of convolution and by the left invariance of the horizontal vector fields.\\\\
Thus , given $y\in B(0,2\sigma_1/k)$, by \eqref{intnablaGrhok},\eqref{nablaGfstarrhok} and the fact that the family $\{U(z\cdot y, \sigma_0/3k):z\in \mathbf Z_k\}$ is disjoint,
\begin{equation*}
\begin{split}
&\sum_{z\in {\bf Z}_k}\int_{U(z\cdot y,\sigma_0/6k)}|\nabla_\G(f\star \rho_k)(x)|^Q\,dx=\\
&=\,\sum_{z\in {\bf Z}_k}\int_{U(z\cdot y,\sigma_0/6k)}\left|\int_{\Rn}\nabla_\G\rho_k(w^{-1})\,f(\tau_x(w))\,dw\right|^Q\,dx\\
&=\,\sum_{z\in {\bf Z}_k}\int_{U(z\cdot y,\sigma_0/6k)}\left|\int_{\Rn}\nabla_\G\rho_k(w^{-1})\,(f(\tau_x(w))-f(x))\,dw\right|^Q\,dx\\
&=\,\sum_{z\in {\bf Z}_k}\int_{U(z\cdot y,\sigma_0/6k)}\left|k^{Q+1}\,\int_{\Rn}\nabla_\G\rho(\delta_k(w^{-1}))\,(f(\tau_x(w))-f(x))\,dw\right|^Q\,dx\\
&=\,\sum_{z\in {\bf Z}_k}\int_{U(z\cdot y,\sigma_0/6k)}\left|k^{Q+1}\,\int_{B(0,\sigma_0/6k)}\nabla_\G\rho(\delta_k(w^{-1}))\,(f(\tau_x(w))-f(x))\,dw\right|^Q\,dx\\
&\le\,c_1\,\sum_{z\in {\bf Z}_k}({\rm osc}_{U(z\cdot y,\sigma_0/3k)} f)^Q\le\,c_1 \,V_Q(f,\Rn)\,.
\end{split}
\end{equation*}
 By integrating the previous inequality with respect to $y\in B(0,2\sigma_1/k)$, it follows that there exists a positive constant $c_2>\,0$ such that
 \begin{equation}\label{estimweakdiff1}
 \begin{split}
 & k^{Q}\,\int_{B(0,2\sigma_1/k)}\left(\sum_{z\in {\bf Z}_k}\int_{U(z\cdot y,\sigma_0/6k)}|\nabla_\G(f\star\rho_k)(x)|^Q\,dx\right)\,dy\\
 &\le\,c_2\, V_Q(f,\Rn)\text{ for each positive integer } k\,.
 \end{split}
\end{equation}
 Notice also that, writing $g(x)=|\nabla_\G(f\star \rho_k)(x)|^Q$, we get, for each $z\in\G,$
\[\begin{split}\int_{U(z,\sigma_1/k)}g(x)\,dx&=\int_{U(0,\sigma_1/k)}g(z\cdot x)\,dx\\ &=c_3 k^Q\int_{U(0,\sigma_0/6k)}\left(\int_{U(0,\sigma_1/k)}g(z\cdot x)\,dx\right)dy\\ &=c_3k^Q\int_{U(0,\sigma_0/6k)}\left(\int_{U(0,\sigma_1/k)\cdot y^{-1}}g(z\cdot x\cdot y)\,dx\right) dy\\ &\leq 
c_3k^Q\int_{U(0,\sigma_0/6k)}\left(\int_{B(0,2\sigma_1/k)} g(z\cdot x\cdot y)\,dx\right) dy\\
&=c_3k^Q\int_{B(0,2\sigma_1/k)}\left(\int_{U(0,\sigma_0/6k)} g(z\cdot x\cdot y)\,dy\right) dx\\
&=c_3k^Q\int_{B(0,2\sigma_1/k)}\left(\int_{U(z\cdot x,\sigma_0/6k)} g(y)\,dy\right) dx.
\end{split}\]
In these passages we performed some changes of variables (which are allowed being the $n$-dimensional Lebesgue measure both left and right invariant), we noticed that $U(0,\sigma_1/k)\cdot y^{-1}\subset B(0,2\sigma_1/k)$ for every $y\in U(0,\sigma_0/6k)$ and applied Fubini-Tonelli theorem.\\
 Let us now notice that, by applying Lemma \ref{estimdistintegers} (ii) and the previous estimates
 , it follows that
 \begin{equation}\label{estimweakdiff2}
 \begin{split}
 &\int_{\Rn}|\nabla_\G(f\star\rho_k)(x)|^Q\,dx\le\,\sum_{z\in{\bf Z}_k}\int_{U(z,\sigma_1/k)}|\nabla_\G(f\star\rho_k)(x)|^Q\,dx\\ 
&\le\,c_3\,k^{Q}\,\int_{B(0,2\sigma_1/k)}\left(\sum_{z\in {\bf Z}_k}\int_{U(z\cdot y,\sigma_0/6k)}|\nabla_\G(f\star\rho_k)(x)|^Q\,dx\right)\,dy\,.
 \end{split}
 \end{equation}
 From \eqref{estimweakdiff1} and \eqref{estimweakdiff2}, we get that there exists $c_4>\,0$ such that
 \begin{equation}
 \int_{\Rn}|\nabla_\G(f\star \rho_k)(x)|^Q\,dx\le\,c_4\,V_Q(f,\Rn)<\,\infty,\,\forall\, k\in\N\,.
 \end{equation}
 Since $W^{1,Q}_{\G}(\Rn)$ is a reflexive space and $f\star\rho_k\to f$ in $L^Q(\Rn)$, as $k\to\infty$, we can infer that $f\in W^{1,Q}_{\G}(\Rn)$. In the general case  $f\in BV^Q_{\G;\rm loc}(\Omega)$, it is sufficient to show that $f\,\vf\in W_\G^{1,Q}(\Omega)$ for each $\vf\in\ci^\infty_c(\Omega)$ which is a cutoff function in a small neighbourhood of a point. Using zero-extension, since $f\,\vf\in BV^Q_\G(\Rn)\cap L^Q(\R^n)$, by the previous argument, $f\,\vf\in W_\G^{1,Q}(\Rn)$ and we are done.
 
 \noindent{\bf 2nd step.} Let us now show the conclusion in the general case. We assume $f\in BV^Q_{\rm loc}(\Omega,\G_2)$ and prove that $f\in W^{1,Q}(\Omega',\G_2)$ for every $\Omega'\subset\subset\Omega.$ 
 Let $u_z:\,\Omega\to\R$ be the function in \eqref{defuz}. Then, by the triangle inequality for distances, it follows that
 \[
 {\rm osc}_{U_{d_1}(x,r)}(u_z)\le\, {\rm osc}_{U_{d_1}(x,r)}(f)
 \]
for each ball $U_{d_1}(x,r)\subset\Omega$ and $z\in\G_2$. In particular, $u_z\in  BV^Q_{\G;\rm loc}(\Omega)$ and
\[
V_Q(u_z,\Omega)\le\,V_Q(f,\Omega)\text{ for each }z\in \G_2\,.
\]
From Theorem \ref{diffthm} and the previous step, $u_z:\,\Omega\to\R$  is P-differentiable a.e. and $u_z\in W_{\G_1,\rm loc}^{1,Q}(\Omega)$ (implying $u_z\in W^{1,Q}_{\G_1}(\Omega')$). Thus, by Proposition \ref{Pweakdiff},
\[
\begin{split}
\left|X_iu_z(x)\right|&=\,\left|\lim_{t\to 0}\frac{u_z\left(\exp(tX_i)(x)\right)-u_z(x)}{t}\right|\\
&\le\,{\rm Lip}(u_z)(x)\le\,\,{\rm Lip}(f)(x)
\end{split}
\]
for a.e. $x\in\Omega$ and for each $z\in\G_2$. Since $\Omega'\subset\subset \Omega$ and $f\in BV^Q_{\rm loc}(\Omega,\G_2)$, by Theorem \ref{diffthm} it follows that ${\rm Lip}(f)\in L^Q(\Omega')$. To conclude notice that $f\in L^Q(\Omega',\G_2)$: indeed  $f\in BV^Q_{\rm loc}(\Omega,\G_2)$, so it is locally bounded and $\Omega'\subset\subset \Omega$.
Hence $f\in W^{1,Q}(\Omega',\G_2)$ according to Definition \ref{SobspacesvaluedCarnot}.\\
\end{proof}

\section{Sobolev maps between Carnot groups with horizontal gradient of sections in $L^{Q,1}$}\label{Steinresult}
In this section we deal with the sharp condition on the summability of the weak gradient of a Sobolev function which ensures continuity and differentiability almost everywhere. 
It is well known that these properties hold for maps in the Sobolev space $W^{1,p}_{\rm loc}(\R^n)$ provided that $p>n$ (in the limit case $p=n$ the answer is negative, see also Remark \ref{remark4.3}). Such a result has been extended to the case of Carnot groups (see Theorem \ref{pdifferent} and Theorem \ref{p>Qdiff}). In his article \cite{SteinAnnMath}, Elias M. Stein sharpened the previous result on Euclidean spaces: if a weakly differentiable function admits weak gradient in the Lorentz space $L^{n,1}(\R^n)$, then it is continuous and differentiable a.e. This is a remarkable improvement since, by Proposition \ref{propLQ1}, $L^{n,1}$ is locally an intermediate space between $L^n$ and $L^p$ (for all $p>n$). The approach of Stein makes use of the machinery of harmonic analysis, while later Kauhanen, Koskela and Malý \cite{KKM} obtained an alternative proof exploiting tools of real analysis and, in particular, the already discussed notion of $n-$absolutely continuous functions. Next we will follow this last strategy in order to discuss Stein's result for maps between Carnot groups.\\
Continuity and differentiability have been already studied also for real valued Sobolev functions on metric measure spaces with upper gradient in the Lorentz space $L^{Q,1}$, where $Q$ denotes the Hausdorff dimension of the metric space (see, for instance, \cite{Romanov,RaMo, WZ3,WZ,WZ2}).  Here we are going to prove that a Sobolev map $f:\,\Omega\subset\G_1\to\G_2$,  with horizontal gradients of its sections $u_z:\,\Omega\to\R$ uniformly bounded in $L^{Q,1}(\Omega)$ (see \eqref{defuz}), actually satisfies the condition (RR), if $Q$ denotes the homogeneous dimension of $\G_1$.

Let us first recall an estimate of the Riesz potential in Carnot groups in terms of the Orlicz conditions, which extends the Euclidean one first introduced in \cite[Theorem 3.1]{KKM}.
\begin{theorem}\label{rieszpotential} Let $\G\equiv\Rn$ be a Carnot group, $d$ the associated CC distance, $Q$ its homogeneous dimension and $\Omega\subset\G$. Let $g$ be a nonnegative measurable function on $\Omega$ and let $\vf$ be a nonincreasing
positive function on $(0,\infty)$. If $F_\varphi (s)$ is defined as in \eqref{Fvf}, then for any $z\in\Omega$ and any measurable set $E\subset\Omega$, we have the inequality
\[
\left(\int_E d(x,z)^{1-Q}g(x)\,dx\right)^Q\le\, C_Q\,\left(\int_0^\infty\vf^{\frac{1}{Q}}ds\right)^{Q-1}\int_E F_\vf(g)\,dx
\]
where $C_Q:=\, \left(1+|\partial U(0,1)|_\G(\G)\right)^{Q}$ and $|\partial U(0,1)|_\G(\G)$ denotes the $\G$-perimeter of the CC open ball $U(0,1)$. 
\end{theorem}
\begin{proof}
The proof follows the Euclidean one in \cite[Theorem 3.1]{KKM}, with suitable changes. We provide here the details for the reader's convenience. 
\\
   Up to a left translation, we can assume $z=0$. Moreover, we can suppose the integral on the left-hand side to be positive, since otherwise there is nothing to prove. For notational simplicity, let's write \begin{equation}\label{I1I2}I_1=\int_0^{+\infty}\varphi^\frac{1}{Q}(s)\,ds,\qquad\qquad I_2=\int_EF_\varphi(g(x))\,dx.\end{equation}
   Similarly as before, we will assume $I_1,I_2<+\infty$, otherwise the proof is over.\\
   Let $0<c<+\infty$ be an arbitrary value such that $c\leq \int_E\|x\|^{1-Q} g(x)\,dx$, where $\|\cdot\|$ denotes the homogeneous norm associated to the distance $d$. \\
   Let $S:=\{(x,s)\in E\times \mathbb R: 0<s<g(x)\}$. We can write, by Fubini-Tonelli theorem
   \[\int_E\|x\|^{1-Q}g(x)\,dx=\int_S\|x\|^{1-Q} \,dx\,ds.\]
   Define \[\lambda=\frac{I_2^{1/Q}}{c\,I_1^{1/Q}}\] 
   and 
   \[S'=\left\{(x,s)\in S: \varphi(s)<\left(\frac{\|x\|}{c\lambda}\right)^Q\right\}\qquad\quad S''=S\setminus S'.\]
   If $(x,s)\in S'$, then $\|x\|^Q>c^Q\lambda^Q\varphi(s)\geq c^Q\lambda^Q\varphi(g(x))$ since $\varphi$ is nonincreasing.\\ Hence we estimate
   \begin{equation}\label{S'}\begin{split}\int_{S'}\|x\|^{1-Q}dx\,ds&\leq c^{1-Q}\lambda^{1-Q}\int_{S'}\varphi^\frac{1-Q}{Q}(g(x))\,dx\,ds\\ &\leq c^{1-Q}\lambda^{1-Q}\int_E\left(\int_0^{g(x)}ds\right)\varphi^\frac{1-Q}{Q}(g(x))\,dx\\ &=c^{1-Q}\lambda^{1-Q}\int_EF_\varphi(g(x))\,dx=c^{1-Q}\lambda^{1-Q}I_2.\end{split}\end{equation}
   Now we bound the integral over $S''$ by using Fubini-Tonelli theorem and the integration formula in polar coordinates in Carnot groups (see, for instance, \cite[Corollary 4.6]{MSC}):
   \begin{equation}\label{S''}\begin{split}\int_{S''}\|x\|^{1-Q}dx\,ds &\leq \int_0^{+\infty}\left(\int_{B(0,c\lambda \varphi^{1/Q}(s))}\|x\|^{1-Q} \,dx\right)\,ds\\&= \int_0^{+\infty}\left(\int_0^{c\lambda \varphi^{1/Q}(s)}\left(\int_{\partial U(0,1)}\|\delta_r(y)\|^{1-Q} r^{Q-1} d|\partial U(0,1)|_\G(y)\right)dr\right)ds\\&\leq |\partial U(0,1)|_\G(\G)\int_0^{+\infty}c\lambda\varphi^\frac{1}{Q}(s)\,ds=|\partial U(0,1)|_\G(\G)\, c\lambda I_1.\end{split}\end{equation}
   Combining (\ref{S'}) and (\ref{S''}), we get
   \[\begin{split}c\leq \int_E\|x\|^{1-Q}g(x)\,dx&=\int_{S'}\|x\|^{1-Q} \,dx\,ds+\int_{S''}\|x\|^{1-Q}\,dx\,ds\\ &\leq c^{1-Q}\lambda^{1-Q}I_2+|\partial U(0,1)|_\G(\G)\,c\lambda I_1=(1+|\partial U(0,1)|_\G(\G))I_1^\frac{Q-1}{Q}I_2^\frac{1}{Q},\end{split}\]
   which implies 
   \[c^Q\leq C_QI_1^{Q-1}I_2=C_Q\bigg(\int_0^{+\infty}\varphi^\frac{1}{Q
   }(s)\,ds\bigg)^{Q-1}\int_EF_\varphi(g(x))\,dx.\]
Since this inequality holds for every $c\leq \int_E\|x\|^{1-Q}g(x)\,dx$, we get the conclusion.
\end{proof}
\begin{theorem}\label{SobmapRR} Let $\G_1\equiv\Rn$, $\G_2\equiv\Rm$, $\Omega\subset\G_1$ be an open set and  let $Q$ denote the homogeneous dimension of $\G_1$. Assume that $f\in W^{1,1}_{\rm loc}(\Omega,\G_2)$ and  the function $g$ in \eqref{estimnablauz} belongs to $ L^{Q,1}(\Omega)$.  Then, a suitable representative of $f$ is continuous and satisfies the condition (RR). In particular $f\in AC^Q(\Omega,\G_2)$ and then  it is P-differentiable a.e. in $\Omega$.
\end{theorem} 
\begin{proof} {\bf 1st step.} First assume $\G_1=\G$, $\G_2=\R$ with $f\in W^{1,1}_{\G,\rm loc}(\Omega)$ and $|\nabla_\G f|\in L^{Q,1}(\Omega)$. We  stress that this case was already studied in the general case of measure metric spaces (see, for instance,  \cite[Thm. 6.6]{WZ}) by extending the strategy exploited in \cite{KKM}. However we are going to repeat this strategy since  we will use it in the second step, which is still open to our knowledge.\\
First of all notice that we can apply Theorem \ref{charLQ1} inferring that, since $|\nabla_\G f|\in L^{Q,1}(\Omega)$, there exists a nonnegative nonincreasing function $\varphi:(0,+\infty)\to[0,+\infty)$ such that
\begin{equation}\label{propphi1}\int_0^{+\infty}\varphi^\frac{1}{Q}(s)\, ds<+\infty,\end{equation}
\begin{equation}\label{propphi2}\int_\Omega F_\varphi(|\nabla_\G f(x)|)\, dx<+\infty.\end{equation}
Let's apply now Theorem \ref{representation}, finding a negligible set $N$ such that 
\begin{equation}\label{rieszformula}|f(z)-f_U|\leq C\int_U\frac{|\nabla_\G f(y)|}{d(z,y)^{Q-1}}\, dy\end{equation}
for each CC open ball $U\subset \Omega$ and $z\in U\setminus N$. This implies that, for every CC open ball $U\subset \Omega$,
\begin{equation}\label{oscU-N}(\text{osc}_{U\setminus N}f)^Q\leq \int_Uw(x)\, dx,\end{equation} where 
\[w(x)=2^Q C^QC_Q\left(\int_0^{+\infty}\varphi^\frac{1}{Q}(s)\,ds\right)^{Q-1}F_\varphi(|\nabla_\mathbb Gf(x)|).\]
In order to prove (\ref{oscU-N}) we can assume the left-hand side to be positive, and choose $0\leq c<\text{osc}_{U\setminus N}f$.
By definition of oscillation and the triangular inequality, there exists $z\in U\setminus N$ such that $c\leq 2|f(z)-f_U|$. Combining this fact with (\ref{rieszformula}) and Theorem \ref{rieszpotential}, we infer that
\[\begin{split}c^Q&\leq 2^Q C^Q\left(\int_U\frac{|\nabla_\G f(y)|}{d(z,y)^{Q-1}}\,dy\right)^Q\\ 
&\leq 2^QC^QC_Q\,\left(\int_0^\infty\vf^{\frac{1}{Q}}ds\right)^{Q-1}\int_U F_\vf(|\nabla_\G f(x)|)\,dx\\ &=\int_Uw(x)\,dx.
\end{split}\]
By arbitrariness of $0\leq c<$osc$_{U\setminus N }u$, we conclude that (\ref{oscU-N}) is satisfied.
As a consequence, $f$ is locally uniformly continuous on $\Omega\setminus N$ and can be extended to a continuous function $\tilde f$ on the whole space $\Omega$ which satisfies the (RR) condition with weight $w$.

\noindent{\bf 2nd step.} Let us now prove the general case. For given $z\in\G_2\equiv\Rm$ let $u_z:\,\Omega\to\R$ denote a section of $f$, that is, the function defined in \eqref{defuz}. By definition $u_z\in\ W^{1,1}_{\G_1, \rm loc}(\Omega)$ and satisfies \eqref{estimnablauz} with $g\in L^{Q,1}(\Omega)$.
Thus, by the 1st step, there exists a continuous function $\tilde u_z:\,\Omega\to\R$ and a negligible set $N_z\subset\Omega$, found as in (\ref{rieszformula}), such that
\begin{equation}\label{tildeuz=uz}
\tilde u_z(x)=\,u_z(x)=\,d_2(z,f(x))\text{ for each }x\in\Omega\setminus N_z\,.
\end{equation}

Now we proceed again as in the 1st step:  since $g\in L^{Q,1}(\Omega)$, we know that there exists a nonnegative nonincreasing function $\varphi:(0,+\infty)\to [0,+\infty)$ such that 
\begin{equation}
    \int_0^{+\infty} \varphi^\frac{1}{Q}(s)\,ds<+\infty,
\end{equation}
\begin{equation}
        \int_\Omega F_\varphi(g(x))\, dx<+\infty.
    \end{equation}
    We want to show that 
    \begin{equation}\label{oscU-Nz}
        (\text{osc}_{U\setminus N_z}\tilde u_z)^Q\leq \widetilde C \left(\int_{0}^{+\infty}\varphi^\frac{1}{Q}(s)\,ds\right)^{Q-1}\int_UF_\varphi(g(x))\,dx
    \end{equation}
    for each CC open ball $U\subset\Omega$ and each $z\in\G_2$.
    If the left-hand side is 0, clearly there is nothing to prove. Otherwise pick any $0\leq c<\text{osc}_{U\setminus N_z}\tilde u_z$. This implies that there exists $x\in U\setminus N_z$ such that
    $c\leq 2|\tilde u_z(x)-(\tilde u_z)_U|$.
    Hence, using Theorem \ref{rieszpotential},
    \[\begin{split}c^Q&\,\leq\,\, 2^Q|\tilde u_z(x)-(\tilde u_z)_U|^Q\stackrel{(\ref{rieszformula})}\leq 2^Q C^Q\left(\int_U\frac{|\nabla_{\mathbb G_1} \tilde u_z(y)|}{d_1(x,y)^{Q-1}}\,dy\right)^Q\\ &\,\leq \,\,2^Q C^Q\left(\int_U\frac{ g(y)}{d_1(x,y)^{Q-1}}\,dy\right)^Q\\
    &\,\leq 2^Q C^Q C_Q\left(\int_0^{+\infty}\varphi^\frac{1}{Q}(s)\,ds\right)^{Q-1}\int_U F_\varphi(g(x))\,dx.
    \end{split}\]
    Since this estimate holds for every $0\leq c<\text{osc}_{U\setminus N_z}\tilde u_z$, we get (\ref{oscU-Nz}).\\
    From the fact that $\tilde u_z$ is continuous on $\Omega$, the same estimate holds replacing osc$_{U\setminus N_z}\tilde u_z$ with osc$_{U}\,\tilde u_z$.
Let $N=\,\cup_{z\in\Q^m}N_z$. Then $N$ is still negligible and, by \eqref{tildeuz=uz}, it follows that
 \begin{equation}\label{oscfrationals}
 \begin{split}
&| d_2(z,f(x))-d_2(z,f(y))|^Q\le\,({\rm osc}_{U}\tilde u_z)^Q\\
&\le\, \widetilde C \left(\int_{0}^{+\infty}\varphi^\frac{1}{Q}(s)\,ds\right)^{Q-1}\int_UF_\varphi(g(x))\,dx
\end{split}
 \end{equation}
 for each $x,\,y\in U\setminus N$ and $z\in\Q^m$, for each CC open ball $U\subset\Omega$. Now observe that, by the triangle inequality and  density of $\Q^m$ in $\G_2\equiv\Rm$, it follows that
 \begin{equation}\label{supoverz}
 \sup_{z\in\Q^m}| d_2(z,f(x))-d_2(z,f(y))|=\,d_2(f(x),f(y))
 \end{equation}
 for each $x,\,y\in U\setminus N$. Thus, from \eqref{oscfrationals} and \eqref{supoverz}, it follows that
 \begin{equation}\label{oscf}
d_2(f(x),f(y))^Q\le \widetilde C\left(\int_{0}^{+\infty}\varphi^\frac{1}{Q}(s)\,ds\right)^{Q-1}\int_UF_\varphi(g(x))\,dx,
 \end{equation}
 for each $x,\,y\in U\setminus N$ and each CC open ball $U\subset\Omega$. 
This implies that $f:\,\Omega\setminus N\to\G_2$ is locally uniformly continuous and, since $\Omega\setminus N$ is dense in $\Omega$, it can be extended  to a (unique) continuous function $f:\Omega\to\G_2$. Moreover, from \eqref{oscf}, it follows also that
 \[
({\rm osc}_U f)^Q\le\,\widetilde C\left(\int_{0}^{+\infty}\varphi^\frac{1}{Q}(s)\,ds\right)^{Q-1}\int_UF_\varphi(g(x))\,dx,
 \]
 for each CC open ball $U\subset\Omega$. 
 This means that $f$ satisfies the condition (RR) with weight \[w(x)=\widetilde C\left(\int_0^{+\infty}\varphi^\frac{1}{Q}(s)\,ds\right)^{Q-1}F_\varphi(g(x)).\qedhere\]
 \end{proof} 

\begin{remark}\label{remark4.3} Let us recall that Theorem \ref{SobmapRR} is sharp already in the Euclidean case. Indeed , Stein \cite{SteinAnnMath}  pointed out that, if $\nabla f\notin L^{n,1}$, then both continuity and differentiability of $f$ may fail.
For instance, an explicit example is given by the function $u(x):=\log(\log(1/|x|))$ on $B(0,1/3)$ extended in a smooth way on $\R^n$ with compact support. By a direct computation it is easy to prove that $u\in W^{1,n}(\R^n)$ and moreover $|\nabla u|\in L^{n,q}$ for every $q>1$ (see \cite{Stein} for the definition of general Lorentz spaces).
If $\{q_k\}_{k\in\mathbb N}$ is an enumeration of the rational points in $\mathbb R^n$, we set $f_k(x)= u(x-q_k)$ and define 
\[f(x)=\sum_{k=0}^{+\infty}2^{-k}f_k(x).\]
Then such a map $f$, and its weak gradient, have the same summability properties of $u$, but it is essentially unbounded (thus discontinuous and not differentiable) around any point. Hence Theorem \ref{SobmapRR} is sharp also in the refined class of Lorentz spaces.
\end{remark}
\begin{remark} Note that Theorem \ref{SobmapRR} is not invertible, in the following sense: even though $\Omega$ is a bounded open set of an Euclidean space $\Rn$, it is well known that a function $f\in AC^n(\Omega)$ need not admit weak gradient $\nabla f$ in $L^{n,1}(\Omega)$ (see, for instance, \cite[Theorem 3.7]{Hencl2}).
\end{remark}
\bigskip
\section{Area formula and change of variables for absolutely continuous functions between Carnot groups}\label{AF}
An important property of Lipschitz functions between Euclidean spaces is the so called area formula, relating the Hausdorff measure of the image of a Lipschitz map (counted with multiplicity) to the integral of the Jacobian on the domain. Such a formula holds also in the setting of Carnot groups (Theorem \ref{areaformulalippansu}); we will extend this result to the case of absolutely continuous functions (see  Theorem \ref{areaformulaAC}).
\subsection{Area formula between Carnot groups} The content of this section mainly relies on \cite{Ma2001}.

An important tool in Euclidean geometric measure theory is the so-called \emph{area formula} for Lipschitz maps, which turns to be true also for Lipschitz  maps between Carnot groups.
\begin{definition}\label{hjacobiandef}[H-Jacobian] Given two Carnot groups $(\G_i,d_i)$, $i=1,2$, endowed with an invariant distance $d_i$, let $L:\,\G_1\to\G_2$ be a $H$- linear map. Let $Q$ denote the homogeneous dimension of $\G_1$. The \emph{horizontal jacobian} of $L$ is defined by
\[
J_Q(L):=\frac{\mathcal H_{d_2}^Q(L(B_1))}{\mathcal H_{d_1}^Q(B_1)}
\]
where $B_1$ denotes the unit open ball in $(\G_1,d_1)$ and for $i=1,2$, $\mathcal H_{d_i}^Q$  denotes the Hausdorff measure w.r.t. $d_i$.
\end{definition}
\begin{remark}In the previous definition it is possible to replace $B_1$ with any other measurable subset of $\G_1$ of positive finite measure, by the $H$-linearity of the map $L$. \end{remark}
The following Theorem is taken from \cite[Theorem 4.4]{Ma2001}.
\begin{theorem}{\rm (Area formula for Lipschitz maps)}\label{areaformulalippansu} Let $f:\,E\subset(\G_1, d_1)\to (\G_2, d_2)$  be a Lipschitz continuous function with $E$ measurable set. Then
\[
\int_{E}J_Q(d_Pf(x))\,d\mathcal H^Q_{d_1}(x)=\,\int_{\G_2}N(f,E,y)\,d\mathcal H^Q_{d_2}(y)
\]
where $N(f,E,\cdot)$ denotes the multiplicity function of $f$ relatively to $E$, that is $N(f,E,y)=\#\{f^{-1}(y)\cap E\}\in\N\cup\{+\infty\}$.
\end{theorem}
\begin{remark}\label{diffonmeasurable}
  The original definition of P-differentiability (cf. Definition \ref{pansudifferentiability}) requires $E$ to be an open subset of a Carnot group $\G_1$. However, it is possible to extend this notion to the set of density points of a general measurable set. In particular, the Pansu's differential, if it exists, is uniquely determined at a density point of $E$. Moreover, a generalization of Theorem \ref{pansuthm} holds: it has been proved by Magnani in \cite[Theorem 3.9]{Ma2001} that a Lipschitz map $f:E\to\G_2$ is $\mathcal H ^Q_{d_1}$ a.e. P-differentiable, where $E\subset \G_1$ is measurable, while $Q$ and $d_1$ denote the homogeneous dimension and the CC distance of $\G_1$ respectively. 
\end{remark}
Aim of this chapter is to extend Theorem \ref{areaformulalippansu} to the case of maps in $AC^Q(\Omega,\G_2)$, which are not necessarily Lipschitz continuous.
The following useful result is proved in \cite[Lemma 2.12]{LZ} in the general setting of metric measure spaces.

\begin{proposition}[Lusin's (N) condition for $ AC^Q$]\label{lusincond}
  Let $\G_1\equiv \Rn$ and $\G_2\equiv \Rm$ be Carnot groups and let $\Omega\subset \G_1$ be an open set. Let $f\in AC^Q(\Omega,\mathbb G_2)$, where $Q$ denotes the homogeneous dimension of $\G_1$. Then, for every measurable subset $E\subset \Omega$, it holds
  \[\mathcal L^n(E)=0\;\Longrightarrow \;\mathcal H^Q(f(E))=0.\]
\end{proposition}
The next result is inspired by Theorem 3.1.9 in \cite{federer} and adapted to our setting.
\begin{theorem}\label{lipschitzpartition}
Let $\Omega\subset \G_1\equiv \Rn$ be an open set and $A\subset \Omega$ be measurable. If $f:\Omega\to \G_2\equiv \Rm$ has the property that
\[\limsup_{y\to x}\frac{d_2(f(y),f(x))}{d_1(y,x)}<+\infty\qquad \text{for all $x\in A$,}\]
then there exists a countable measurable partition $\{A_k\}_k$ of $A$ such that $f_{\upharpoonright A_k}$ is Lipschitz continuous.
\end{theorem}
\begin{proof}
For every $j\in\N$, we define
\begin{equation}\label{defcj}C_j:=\{z\in\Omega: d_2(f(y),f(z))\leq jd_1(y,z)\quad\text{for every }y\in U_{d_1}(z,1/j)\}.\end{equation}
It follows that $A$ is contained in $\cup_j C_j$. 
Indeed, let $a\in A,$ so that
\[\limsup_{y\to a}\frac{d_2(f(y),f(a))}{d_1(y,a)}<+\infty.\]
This implies that there exists $j\in \N$ and $\bar r>0$ sufficiently small such that $\frac{d_2(f(y),f(a))}{d_1(y,a)}<j$ for every $y\in U_{d_1}(a,\bar r)$. Hence, up to increase $j$, we deduce that the previous inequality holds for every $y\in U_{d_1}(a,\frac{1}{j})$, which means that $a\in C_j$.\\ 
Notice that each $C_j$ is closed in $\Omega$. In fact, let $x_n\in C_j$ which converge to $x\in \Omega$ as $n\to\infty$ and let $y\in U_{d_1}(x,\frac{1}{j})$. If $n\in\N$ is chosen sufficiently large, we can assume $\{x, y\}\subset U_{d_1}(x_n,1/j)$. Hence
\[\begin{split}d_2(f(y),f(x))&\leq d_2(f(y),f(x_n))+d_2(f(x_n),f(x))\\
&\leq jd_1(y,x_n)+jd_1(x_n,x)\end{split}\]
and, passing to the limit as $n\to\infty$, we get $d_2(f(y),f(x))\leq jd_1(y,x)$, so that $x\in C_j$.\\
Exploiting the Lindel\"of property of $\G_1\equiv \Rn$, we can write $C_j$ as a countable union of closed sets $C_j^i$, for $ i\in \N$, each of which of diameter less than $\frac{1}{j}$. If we relabel $\{C_k\}_k=\{C_j^i\}_{j,i}$, it follows that $f_{\upharpoonright C_k}$ is Lipschitz by \eqref{defcj}. To conclude, define $A_k:=C_k\cap A$ which are still measurable and notice that we can suppose the sets $A_k$ to be disjoint (hence a partition of $A$) by the usual differences procedure.
\end{proof}
We now combine Theorem \ref{lipschitzpartition} and Theorem \ref{areaformulalippansu} to get an area formula for absolutely continuous maps between Carnot groups.
\begin{theorem}\label{areaformulaAC}
  Let $\Omega\subset \G_1$ be open and $f\in AC^Q(\Omega,\G_2)$. For every measurable subset $E$ of $\Omega$, 
  we have
  \begin{equation}\label{areaformulaAC1}\int_EJ_Q(d_Pf(x))\,d\mathcal H^Q_{d_1}(x)=\int_{\G_2} N(f,E,y)\,d\mathcal H^Q_{d_2}(y)\end{equation}
  where $J_Q$ is the Jacobian introduced in Definition \ref{hjacobiandef}, while $N(f,E,\cdot)$ denotes the multiplicity function of $f$ relatively to $E$.
\end{theorem}
\begin{proof}
 We note that, by Theorem \ref{diffthm}, Lip$(f)$ is finite almost everywhere, which means
 \[\text{Lip}(f)(x):=\limsup_{y\to x}\frac{d_2(f(y),f(x))}{d_1(y,x)}<+\infty \quad for\,\,\mathcal L^n-a.e.\,x\in\Omega.\]
 We can therefore apply Theorem \ref{lipschitzpartition} and write $\Omega$ as a disjoint union of measurable subsets\[\Omega=A_0\cup \bigcup _{k=1}^{+\infty}A_k,\]
 with $\mathcal L^n(A_0)=0$ and $f_{\upharpoonright A_k}$ Lipschitz.\\ Moreover we know, again by Theorem \ref{diffthm}, that $f$ is P-differentiable almost everywhere. In the negligible set where $f$ does not satisfy this condition, we can define $d_Pf(x)$ to be for instance the zero map.\\
 Let now $E$ be a measurable subset of $\Omega$ and denote $E_k:=E\cap A_k$. Applying Theorem \ref{areaformulalippansu} we get, for every $k\geq 1$,
 \[\int_{E_k}J_Q(d_Pf_{\upharpoonright E_k}(x))\,d\mathcal H^Q_{d_1}(x)=\int_{\G_2} N(f,E_k,y)\,d\mathcal H^Q_{d_2}(y).\]
 Moreover this identity trivially holds for $k=0$ in view of Lemma \ref{lusincond}. Note that, if $x$ is a density point of $E_k$ and $f:\Omega\to\G_2$ is P-differentiable at $x\in E_k$, then also $f_{\upharpoonright E_k}$ is P-differentiable at $x$ in the sense of Remark \ref{diffonmeasurable} and the two differentials coincide. Hence, summing over $k\in\N$ we find, by a simple application of monotone convergence,
\[\begin{split}&\int_{E}J_Q(d_Pf(x))\,d\mathcal H^Q_{d_1}(x)=\sum_{k=0}^{+\infty}\int_{E_k}J_Q(d_Pf_{\upharpoonright E_k}(x))\,d\mathcal H^Q_{d_1}(x)\\ &=\sum_{k=0}^{+\infty}\int_{\G_2} N(f,E_k,y)\,d\mathcal H^Q_{d_2}(y)=\int_{\G_2}N(f,E,y)\,d\mathcal H^Q_{d_2}(y).\qedhere\end{split}\]
 \end{proof}
 

Through standard arguments of approximation by simple functions, it is easy to extend formula \eqref{areaformulaAC1} to a more general form:
\begin{corollary}\label{areaformula2}
 If $f\in AC^Q(\Omega,\G_2)$, then for any measurable subset $E$ of $\Omega$ and each summable function $u: E\to\R$, it holds
 \[\int_E u(x)J_Q(d_Pf(x))\,d\mathcal H^Q_{d_1}(x) = \int_{\G_2}\sum_{x\in f^{-1}(y)}u(x)\, d\mathcal H^Q_{d_2}(y).\]
\end{corollary}
 As a straightforward consequence of Theorem \ref{SobmapRR} and Theorem \ref{areaformulaAC}, we get the following
 \begin{corollary}
     Let $\Omega\subset\G_1$ be an open set and let $f\in W^{1,1}_{\rm loc}(\Omega,\G_2)$ such that the function $g$ in \eqref{estimnablauz} belongs to $ L^{Q,1}(\Omega)$. Then a suitable representative of $f$ satisfies
      \[\int_EJ_Q(d_Pf(x))\,d\mathcal H^Q_{d_1}(x)=\int_{\G_2} N(f,E,y)\,d\mathcal H^Q_{d_2}(y)\]
      for every measurable subset $E\subset\Omega$. Moreover, for every $u:E\to \R$ summable,
       \[\int_E u(x)J_Q(d_Pf(x))\,d\mathcal H^Q_{d_1}(x) = \int_{\G_2}\sum_{x\in f^{-1}(y)}u(x)\, d\mathcal H^Q_{d_2}(y).\]
 \end{corollary}
\section{Appendix: Orlicz-Sobolev spaces on Carnot groups}\label{appendix6}

Orlicz spaces arise naturally in many applications in real, harmonic or functional analysis. Such spaces were introduced by the Polish mathematician Władysław Orlicz in 1932 and are generalizations of the classical $L^p$ spaces. Indeed, the role played by the function $t^p$ is replaced by more general convex maps, called $N$-functions.\\
We recall now the main definitions and basic properties of Orlicz spaces: see for instance \cite{adams},\cite{kranosel} or \cite{luxemburg} for an exhaustive treatment. 
\begin{definition}
    A map $A:[0,+\infty)\to[0,+\infty)$ is called an {\em N-function} if it is convex, vanishes only at $0$ and has the following limit properties:
    \[\lim_{t\to 0^+}\frac{A(t)}{t}=0,\qquad \lim_{t\to\infty}\frac{A(t)}{t}=+\infty.\]
\end{definition}
It is easy to verify that a map satisfying these properties is actually continuous on $[0,+\infty)$ and strictly increasing.
Given an $N$-function $A:[0,+\infty)\to[0,+\infty)$, its \textit{Young conjugate} $\widetilde A$ is defined as the $N$-function
\begin{equation*}\label{young}\widetilde A(t):=\sup_{s>0}\,\{st-A(s)\}.\end{equation*}
Examples of $N$-functions and associated conjugates are for instance
\[\begin{split}A(t)=\frac{t^p}{p},&\qquad \widetilde A(t)=\frac{t^{p'}}{p'}\qquad\left(1<p<+\infty,\quad p'=\frac{p}{p-1}\right);\\ 
A(t)=e^t-&t-1, \qquad \widetilde A(t)=(1+t)\log(1+t)-t.\end{split}\]
\begin{definition}
    Let $\Omega\subset \mathbb R^n$ be measurable and let $A$ be an $N-$function. For every measurable map $u:\Omega\to\mathbb R$, we define its {\em Luxemburg norm} as
    \[\|u\|_{L^A(\Omega)}:=\inf\left\{\lambda>0: \int_\Omega A\left(\frac{|u(x)|}{\lambda}\right)\,dx\leq 1\right\}.\]
    The {\em Orlicz space} associated to the $N$- function $A$ is defined as
    \[L^A(\Omega):=\{u:\Omega\to\mathbb R \text{ measurable}: \|u\|_{L^A(\Omega)}<+\infty\}.\]
    
\end{definition}
Equipped with $\|\cdot\|_{L^A(\Omega)}$, the Orlicz space $L^A(\Omega)$ is a Banach space. Moreover, 
an interesting relation between an $N-$function and its Young conjugate is expressed by the following Hölder-type inequality:
\[\int_\Omega u(x)v(x)\,dx\leq 2\|u\|_{L^A(\Omega)}\|v\|_{L^{\widetilde A}(\Omega)}.\]
Notice that, when considering $A(t):=t^p$, for $1<p<+\infty$, we get that 
$L^A(\Omega)$ coincides with the classical $L^p(\Omega)$ space. Moreover in this case we get $\|\cdot\|_{L^A(\Omega)}=p^{-1/p}\|\cdot\|_{L^p(\Omega)}$. 
\\\\
In order to establish an inclusion relation similar to the classical one for Lebesgue spaces on domain with finite measure, we introduce two new concepts: if $A$ and $B$ are $N$-functions, we will say that \textit{$B$ dominates $A$ globally} if there exists a positive constant $k$ such that $A(t)\leq B(kt)$ for all $t\geq 0$. If this estimate holds only for large values of $t$, we will say that \textit{$B$ dominates $A$ near infinity}.
An interesting result is that the embedding $L^B(\Omega)\hookrightarrow L^A(\Omega)$ holds if and only if either $B$ dominates $A$ globally or $B$ dominates $A$ near infinity and $|\Omega|<+\infty.$
Hence functions which are globally equivalent, i.e. each one dominates the other globally, determine the same Orlicz space (and clearly the same holds if we require that they are equivalent near infinity and $|\Omega|<+\infty$).\\
Note that, if $1<p<q<+\infty$, then the map $B(t):=t^q$ dominates the function $A(t):=t^p$ near infinity, so that $L^q(\Omega)\subset L^p(\Omega)$ if $\Omega$ has finite measure.
\\\\
In the same spirit of Sobolev space, it is now straightforward to define Orlicz-Sobolev functions by means of Orlicz integrability conditions on the weak derivatives. We give this definition directly in the general setting of Carnot groups:
\begin{definition}\label{orlsobspace}
Let $\Omega\subset\mathbb G\equiv \mathbb R^n$ be open and let $A$ be an $N$-function.
We define the {\em Orlicz-Sobolev space} 
\[W^{1,A}_\mathbb G(\Omega):=\{u\in L^A(\Omega): \exists X_iu\in L^A(\Omega)\quad i=1,\dots,m_1\},\]
where $X_i u$ are the distributional derivatives of $u$ computed along the generating vector fields of the group.
We also set 
\[\|u\|_{W^{1,A}_\mathbb G(\Omega)}:=\|u\|_{L^A(\Omega)}+\|\nabla_\mathbb Gu\|_{L^A(\Omega)}.\]
\end{definition}
Equipped with such a norm, it is easy to see that $(W^{1,A}_\G(\Omega),\|\cdot\|_{W^{1,A}_\mathbb G(\Omega)})$ is a Banach space. If $\G$ is the Euclidean $\R^n$, we will use the notation $W^{1,A}(\Omega)$.
\\\\
The classical embedding theorems for Sobolev spaces have been extended also to Orlicz-Sobolev spaces in the Euclidean setting. In particular, the condition $p>n$ has to be replaced by a more general requirement on the growth of the $N$-function which induces the space. In \cite{adams} or \cite{donaldson} it has been proven that, if $\Omega\subset\R^n$ is sufficiently regular and $A$ is an $N$-function, then $W^{1,A}(\Omega)$ is embedded in the space of continuous and bounded maps provided that the condition
\begin{equation}\label{adamsassum}\int_1^{+\infty}\frac{A^{-1}(t)}{t^{1+1/n}}\,dt<+\infty\end{equation}
is satisfied. 
In the case of $A(t)=t^p$, we recover the classical embedding theorem for $p>n$. 
However, as remarked in \cite{cianchi}, this condition is not sharp, since it can be improved by requiring 
\begin{equation}\label{sharp}\int_1^{+\infty}\frac{\widetilde A(t)}{t^{1+n'}}\,dt<+\infty,\end{equation}
which is a weaker assumption than \eqref{adamsassum} (see \cite{talenti2}). 
Indeed, G. Talenti proved in the same work that \eqref{sharp} is a sufficient condition for boundedness (see also \cite{cianchi2}), while A. Cianchi showed in \cite{cianchi} that it is also the sharp condition for continuity. Moreover the same author proved in \cite{AlbCia} that this condition is also the correct one to ensure differentiability almost everywhere.\\
The strategy of Cianchi relies on rearrangement tools, while in the following we will provide an alternative way to extend this result in the setting of Carnot groups, following the approach introduced by Kauhanen, Koskela and Malý  in \cite{KKM}, in the same spirit of Theorem \ref{SobmapRR}.
\\
In order to reach this result, we need first to extend Definition \ref{orlsobspace} to the case of group-valued maps, inspired by Definition \ref{SobspacesvaluedCarnot}. This generalization is not straightforward as for Euclidean spaces (in which it suffices to work in components), since a Carnot group cannot be seen in general as direct product of subgroups.
\begin{definition}
    Let $\mathbb G_1\equiv \mathbb R^n, \mathbb G_2\equiv \mathbb R^m$ be Carnot groups. Let $\Omega\subset\mathbb G_1$ be open and let $A$ be an $N$-function.
    \begin{enumerate}
    \item [(i)]We say that a function $u:\Omega\to\mathbb G_2$ belongs to the space $L^A(\Omega,\mathbb G_2)$ if $u:\Omega\subset \mathbb R^n\to\mathbb R^m$ is measurable and the map $u_0:\Omega\to\mathbb R$ defined by $u_0(x)= d_2(0,u(x))$ is in $L^A(\Omega)$;
    \item[(ii)] We say that a function $u:\Omega\to\mathbb G_2$ belongs to the Orlicz-Sobolev space $W^{1,A}(\Omega,\mathbb G_2)$ if
    \begin{itemize}
        \item $u\in L^A(\Omega,\mathbb G_2)$;
        \item for all $z\in\mathbb G_2$, the functions $u_z:\Omega\to\mathbb R$ defined by
        \begin{equation}\label{sectionss}u_z(x):=d_2(z,u(x))\quad \text{if }x\in\Omega\end{equation}
        are in $W^{1,A}_{\mathbb G_1,\rm loc}(\Omega)$ (according to Definition \ref{orlsobspace});
        \item There exists a nonnegative function $g\in L^A(\Omega)$ (which does not depend on $z$) such that
        \begin{equation*}\label{uppgrad2}|\nabla_{\mathbb G_1}u_z|\leq g \quad\text{a.e. in }\Omega, \;\; \text{for each }z\in\mathbb G_2.\end{equation*}
    \end{itemize}
    \end{enumerate}
\end{definition}

We provide now the natural extension of Theorem \ref{p>Qdiff} and Theorem \ref{SobmapRR} to the case of Orlicz-Sobolev maps: under a sufficient growth condition on $A$, functions in the space $W^{1,A}(\Omega,\G_2)$ are $Q$-absolutely continuous and so enjoy nice regularity properties. 

The following Lemma represents the key step in order to apply the strategy proposed by the authors in \cite{KKM} also to the case of Orlicz-Sobolev spaces.
\begin{lemma}\label{existenceofphi}
 Let $\Omega\subset \mathbb G\equiv \mathbb R^n$ be open with $\mathcal L^n(\Omega)<+\infty$.
 Let $A$ be an N-function satisfying
 \begin{equation}\label{integrability}\int_1^{+\infty}\left(\frac{t}{A(t)}\right)^\frac{1}{Q-1}\,dt<+\infty,\end{equation}
 and let $g$ be a nonnegative function in the space $ L^A(\Omega)$.
 Then there exists a nonincreasing function $\varphi:(0,+\infty)\to(0,+\infty)$ such that
 \begin{equation}\label{firstcond}
 \int_\Omega F_\varphi(g(x))\,dx<+\infty
 \end{equation}
 and
 \begin{equation}\label{secondcond}
     \int_0^{+\infty}\varphi^{1/Q}(t)\,dt<+\infty.
 \end{equation}
\end{lemma}
\begin{proof}
Notice first that we can assume that
\begin{equation}\label{integrability2}
    \int_0^1\left(\frac{t}{A(t)}\right)^\frac{1}{Q-1}\,dt<+\infty,
\end{equation}
since it is always possible to replace $A$ with another $N-$function which is equivalent near infinity (so that it determines the same Orlicz space) and decays sufficiently slow to 0.\\
    Since $g\in L^A(\Omega)$, there exists for sure $\bar\lambda>0$ such that $\int_\Omega A(\bar\lambda g(x))\,dx<+\infty$.
    Define $\varphi:(0,+\infty)\to(0,+\infty)$ as
    \[\varphi(t):=\left(\frac{A(\bar\lambda t)}{t}\right)^\frac{Q}{1-Q}.\]
    By definition $\varphi$ is positive and also nonincreasing.  In fact, by convexity of $A$, 
    we know that $A(\lambda t_1+(1-\lambda)t_2)\leq \lambda A(t_1)+(1-\lambda)A(t_2)$ for every $\lambda\in(0,1),\, t_1,t_2\in[0,+\infty).$
    Consider $0< t<s$ and apply this inequality with $t_1=\bar\lambda s,\, t_2=0,\, \lambda=\frac{t}{s}$.
    We get $A(\bar\lambda t)\leq \frac{t}{s}A(\bar\lambda s)$ and, by simple rearrangements, $\varphi(t)\geq \varphi(s)$.
    Moreover we can calculate
    \[\begin{split}\int_\Omega F_\varphi(g(x))\,dx&=\int_{\{g>0\}}g(x)\left[\left(\frac{A(\bar\lambda g(x))}{g(x)}\right)^\frac{Q}{1-Q}\right]^\frac{1-Q}{Q}\, dx\\ &=\int_\Omega A(\bar\lambda g(x))\, dx<+\infty
    \end{split}\]
    and also, exploiting \eqref{integrability} and \eqref{integrability2},
\[\int_0^{+\infty}\varphi^{1/Q}(t)\,dt=\int_0^{+\infty}\left(\frac{A(\bar\lambda t)}{t}\right)^\frac{1}{1-Q}dt=\bar\lambda^\frac{Q}{1-Q}\int_0^{+\infty}\left(\frac{t}{A(t)}\right)^\frac{1}{Q-1}\,dt<+\infty.\qedhere\]
\end{proof}
\begin{remark} It is proved in \cite{cianchi2} that condition \eqref{integrability} is indeed equivalent to the following:
\begin{equation}\label{int3}
    \int_0^{+\infty}\frac{\widetilde A(t)}{t^{1+Q'}}\,dt<+\infty,
\end{equation}
which clearly extends \eqref{sharp} in the case of Carnot groups of homogeneous dimension $Q$.
\end{remark}
Under condition \eqref{integrability}, or equivalently \eqref{int3}, functions in the space $W^{1,A}(\Omega,\G_2)$ are $Q-$absolutely continuous. Actually, only an Orlicz bound on the weak horizontal gradient of the sections is needed. This is expressed by the following result.

\begin{theorem}\label{finale}
    Let $\mathbb G_1\equiv \mathbb R^n$, $\mathbb G_2\equiv\mathbb R^m$ be Carnot groups and let $\Omega\subset\mathbb G_1$ open. Suppose also that $\mathcal L^n(\Omega)<+\infty$ and denote by $Q$ the homogeneous dimension of $\mathbb G_1$. Let $A$ be an N-function satisfying \eqref{integrability}.
    If $f\in W^{1,1}_{\rm loc}(\Omega,\mathbb G_2)$ and there exists a function $g\in L^A(\Omega)$ such that
    \[|\nabla_\mathbb G u_z|\leq g\qquad \text{for all }z\in\mathbb G_2,\]
    where $u_z$ are defined as in \eqref{sectionss},
    then there exists a representative of $f$ which is continuous and satisfies the (RR) condition. In particular $f\in AC^Q(\Omega,\mathbb G_2)$ and is P-differentiable a.e. in $\Omega$.
\end{theorem}
\begin{proof}
The proof of this result can be achieved using the same strategy exploited in Theorem \ref{SobmapRR}. In particular the existence of a nonincreasing function $\varphi:(0,+\infty)\to(0,+\infty)$ satisfying properties \eqref{firstcond} and \eqref{secondcond} is ensured by Lemma \ref{existenceofphi} instead of Theorem \ref{charLQ1}. The other steps of the proof are unchanged.
\end{proof}
\begin{remark}
Equivalently one can observe that, combining Lemma \ref{existenceofphi} and the characterization given by Theorem \ref{charLQ1}, a function $g$ as in the assumptions belongs actually to the Lorentz space $L^{Q,1}(\Omega)$ and then apply Theorem \ref{SobmapRR}.
\end{remark}
We deduce that, under condition \eqref{integrability}, Orlicz-Sobolev maps between Carnot groups are actually in $AC^Q(\Omega,\G_2)$: hence they share properties like continuity, $P$-differentiability a.e. (Theorem \ref{diffthm}) and area formula (Theorem \ref{areaformulaAC} and Corollary \ref{areaformula2}).

\section {Appendix: Differentiation for maps between metric spaces and  metric space-valued Sobolev maps}\label{DiffMS}
\subsection{Differentiation for maps between metric spaces: }\label{DiffMSemb}
In this section we will assume that the metric measure space $X\equiv (X,d_X,\mu)$ admits a $p$-Poincar\'e inquality with $p\ge\,1$, the ambient measure $\mu$ is doubling and  $Y\equiv (Y,d_Y)$ is a (separable) metric space.\\
Let us recall that a general differentiation theory has been carried out   for  maps $f:X\to Y$, whenever $Y$ can be isometrically embedded into a Banach space $V$  satisfying the {\it Radon-Nikodym property} (RNP) (see \cite{CK1,CK2}). We recall  that a Banach space  $V$ sarisfies (RNP) if each Lipschitz function $f:\,\R\to V$ is differentiable a.e.    The first pioneering result is due to Cheeger\cite{cheeger} when $Y=\,\R$, who was able to endow the metric measure space $X$  by  a so-called {\it Cheeger differentiable structure} and proved a Rademacher differentiability theorem for Lipschitz function in this context.  This result was remarkably extended by Cheeger and Kleiner to maps with target space $Y=V$  satisfying (RNP). 
 Then, whenever the target space $Y$  can be bi-Lipschitz embedded into a Banach space $V$ satisfying (RNP), a differentiation theory can be carried out  also for maps $f:\Omega\subset X\to Y\subset V$  and one can get  a Stepanov's type theorem (see, for instance, \cite{WZ2}). Moreover metric space-valued Sobolev functions can be introduced  by means of the notion of {\it upper gradient} in this setting as in  \cite[Chap. 7]{HKST}).\\
This approach may not work, as concern the differentiability issue, if the target space is a Carnot group.  Indeed, as showed in \cite[Theorem 6.1]{CK1}, the simplest Carnot group $\H^1$ does not admit a bi-Lipschitz embedding into any Banach space satisfying (RNP).\\
\subsection{Metric space-valued Sobolev maps: }\label{MSvaluedSM}
As far as the notion of metric space-valued Sobolev mapping is concerned, different  equivalent approaches can be introduced. Let us point out here two ones related to our arguments: the approach by {\it embedding of the target space}  and the one by {\it sections of  mappings}. From now on we can weaken the assumptions on the involved metric spaces. More precisely, we assume that $X\equiv (X,d_X,\mu)$ is a metric measure space with ambient measure $\mu$ {\it finite on bounded sets}, that is $\mu(E)<\,\infty$ for each bounded $E\subset X$, and $Y\equiv (Y,d_Y,y_0)$ is a pointed complete and separable metric space. Moreover, we will denote by $B(x,r)$ and $U(x,r)$ the closed and open ball of $X$ centered at $x$ with radius $r$, respectively.
\vskip5pt
{\bf Approach by embedding of the target metric space:} This approach extends the one due to Shanmugalingam \cite{Shan}
who introduced the notion of {\it Newtonian Sobolev space} $N^{1,p}(X)$ for real-valued functions defined on $X$ and it was carried out in \cite[Chap. 7]{HKST}. Let us briefly recall here the main notions and results. We first introduce vector-valued Sobolev functions on metric spaces. If $V$ is a Banach space and $1\le\,p<\infty$, we say that a mapping $f:X\to V$ belongs to the {\it Dirichlet space} $D^{1,p}(X:V)$ if $f$ is measurable and it possesses a $p$-integrable $p$-weak {\it upper gradient} (see \cite[Section 6.2]{HKST}). When $V=\R$, we simply denote $D^{1,p}(X):=\,D^{1,p}(X:\R)$. It is clear that $D^{1,p}(X:V)$ is a vector space.\\ We denote as {\it Sobolev space of $V$-valued functions on $X$} the space $N^{1,p}(X:V):=\,D^{1,p}(X:V)\cap L^p(X,V)$ and  $N^{1,p}(X:V)$ turns out to be a Banach space with respect to the norm 
\[
\|u\|_{N^{1,p}(X:V)}:=\,\|u\|_{L^p(X,V)}+\|\rho_u\|_{L^p(X)},
\]
where $\rho_u$ denotes the minimal $p$-weak upper gradient of $u$ guaranteed by \cite[Chap. 6 and Theorem 6.3.20]{HKST}. \\ Consider now the {\it Kuratowski isometric embedding} of $Y$ into $V:=\,\ell^\infty(Y)$, with respect to the base point $y_0$. More precisely, let  $\ell^\infty(Y)$ denote the Banach space of bounded functions from $Y$ to $\R$, equipped with the norm
\[
\|\varphi\|_{\ell^\infty(Y)}:=\,\sup_{y\in Y}|\vf(y)|\quad\text{ if }\vf\in 
\ell^\infty(Y)\,.
\]
Then the mapping
\[\begin{split}
Y\ni y\mapsto\,&\vf_y\in\ell^\infty(Y)\\ &\vf_y(z):=\,d(z,y)-d(y_0,z)\quad z\in Y\end{split}
\]
is an isometric embedding (see \cite[Theorem 4.1]{HKST}).
According to \cite{HKST}, we can introduce the Sobolev and Dirichlet  classes of mappings between metric spaces $X$ to $Y$ (depending on the base point $y_0$) as follows:
\[
N^{1,p}(X:Y):=\left\{ f\in N^{1,p}(X:\ell^\infty(Y)): f(x)\in Y\,\mu\text{-a.e. }x\in X\right\}
\]
and similarly
\[
D^{1,p}(X:Y):=\left\{ f\in D^{1,p}(X:\ell^\infty(Y)): f(x)\in Y\,\mu\text{-a.e. }x\in X\right\}\,.
\]
Moreover we denote as $N_{\rm loc}^{1,p}(X:Y)$ ($D_{\rm loc}^{1,p}(X:Y)$, respectively) the class of all maps $f:\,X\to \ell^\infty(Y)$ such that each $x\in X$ has a small radius $r>\,0$  for which $f|_{U(x,r)}\in N^{1,p}(U(x,r):Y)$ ($f|_{U(x,r)}\in D^{1,p}(U(x,r):Y)$, respectively). Eventually we denote as $L^p_{\rm loc}(X,\mu)$ the class of measurable maps $f\to\bar \R$ such that $\int_E |f|^p\,d\mu<\infty$ for each bounded set $E\subset X$.
Let us recall the following characterizations of $N^{1,p}(X:Y)$ and $D^{1,p}(X:Y)$.
\begin{theorem}\label{charsect}(\cite[Theorem 7.1.20]{HKST}) Let $V$ be a Banach space and let $f:\,X\to V$ be a measurable function. The the following conditions are equivalent:
\begin{itemize}
\item [(i)] $f$ has a $\mu$-representative in the Dirichlet class $D^{1,p}(X:V)$;
\item[(ii)] there esists a $p$-integrable Borel function $\rho:\,X\to [0,\infty]$ with the following property: for each $1$-Lipschitz function $\vf:V\to\R$ there esists a $\mu$-representative $f_\vf$ of the function $\vf\circ f$ in $D^{1,p}(X)$ so that the minimal upper gradient $\rho_{f_\vf}$ of $f_\vf$ satisfies $\rho_{f_\vf}\le\,\rho$ $\mu$-a.e. in $X$.
\end{itemize}
\end{theorem}
\begin{theorem}(\cite[Proposition 7.1.38]{HKST})\label{charND} A measurable map $f:X\to Y$ is in  $D^{1,p}(X:Y)$ if and only if there is a non-negative $p$-integrable Borel function $\rho$ on $X$ such that whenever $\gamma:\,[0,\ell(\gamma)]\to X$ is a non constant rectifiable curve in $X$
\begin{equation*}
  d_Y(f\left(\gamma(\ell(\gamma))\right),f\left(\gamma(0)\right))\le\,\int_{\gamma} \rho\,ds\,,
\end{equation*}
where $\ell(\gamma)$ denotes the length of $\gamma$ and
\[
\int_{\gamma} \rho\,ds:=\,\int_0^{\ell(\gamma)} \rho (\gamma(t))\,dt\,.
\]
Moreover, if $\mu(X)$ is finite, then $f\in N^{1,p}(X:Y)$ if and only if  $f\in D^{1,p}(X:Y)$ and the function $X\ni x\mapsto d_Y(y_0,f(x))$ is $p$-integrable for some (and hence every ) $y_0\in Y$.
\end{theorem}
Theorem \ref{charND} shows that there is an intrinsic way of determining
membership in $D^{1,p}(X:Y)$ and $N_{\rm loc}^{1,p}(X:Y)$ that is independent of the
embedding of $Y$ into a Banach space. Instead, membership in $N^{1,p}(X:Y)$ (as well as in $L^p(X:Y))$ depends on the specific embedding and, in particular, on the base point $y_0$.\\
\vskip5pt
{\bf Approach by sections of mappings: }
It is an extension of the one exploited in Definition \ref{SobspacesvaluedCarnot}, first introduced by Ambrosio \cite{ambrosio2}, for the case of $BV$ functions, and then by Reshetnyak \cite{Resh}, for the case of Sobolev functions. More precisely, first  one provides  the notion of a suitable real-valued Sobolev space $W^{1,p}(X)$ and, for each $u\in W^{1,p}(X)$  a notion of modulus of the ``weak gradient" $|\nabla_X u |$ is introduced. For instance,    by using the previous Newtonian Sobolev space $N^{1,p}(X)$,   the minimal $p$-weak upper gradient $\rho_u$ plays the role of $|\nabla_X u |$ (see, for instance, \cite{AG,AILP}). In the following we assume that  $W^{1,p}(X)=\,N^{1,p}(X)$. Then,  given a mapping $f:X\to Y$, we say that $f\in W^{1,p}(X,Y)$ if
\begin{itemize}
\item the map $X\ni x\mapsto d_Y(y_0,f(x))$ belongs to $L^p(X,\mu)$;
\item for each $y\in Y$, the section $u_y:\,X\to\R$, 
\begin{equation}\label{sectfMS}
u_y(x):=\,d_Y(y,f(x))\text{ if } x\in X \text{ belongs to $D^{1,p}(X)$}
\end{equation}
and there exists $g\in L^p(X,\mu)$ such that 
\begin{equation}\label{domugsect}
\rho_{u_y}\le g\,\,\mu\text{-a.e. in $X$, for each $y\in Y$}\,,
\end{equation}
denoting by $\rho_u$ the minimal $p$-weak upper gradient of $u\in D^{1,p}(X)$. A function $g$ as in \eqref{domugsect} will be called a {\it $p$-weak upper gradient of $f$}.
\end{itemize}
As usual, let us denote by $W^{1,p}_{\rm loc}(X,Y)$ the class of all maps $f:\,X\to Y$ such that each $x\in X$ has a bounded open neighborhood $U_x$ for which $f|_{U_x}\in W^{1,p}(U_x,Y)$.\\

As a consequence of Theorem \ref{charsect}, for fixed $y_0$, we get the following coincidence of spaces:
\begin{theorem}\label{W=N} Let $1<\,p<\,\infty$. Then
$W^{1,p}(X,Y)=\,N^{1,p}(X:Y)$.
\end{theorem}
Before the proof of Theorem \ref{W=N}, we will introduce a preliminary result, which is an extension  of   \cite[Proposition 2.31]{KMX} to metric spaces. 
\begin{lemma}\label{CRlipW}Let $1<\,p<\,\infty$,  $f\in W^{1,p}(X,Y)$ and   let $\vf:\,Y\to\R$ be Lipschitz. Then 
\begin{itemize}
\item [(i)] $X\ni x\mapsto\vf(f(x))-\vf(y_0)$ is in $L^p(X,\mu)$.
\item [(ii)] If $f_\vf:=\,\vf\circ f$, then $f_\vf-\varphi(y_0)\in N^{1,p}(X)$ and $f_\vf$ admits a minimal $p$-weak upper gradient $\rho_{f_\vf}\le L\,g$ $\mu$-a.e. in $X$, being $g$ the function in \eqref{domugsect} and $L:=\,Lip(\vf)$.
\end{itemize}
\end{lemma}
\begin{proof} 
(i) By definition of $W^{1,p}(X,Y)$, the map  $X\ni x\mapsto d_Y(y_0,f(x))$ is in $L^p(X,\mu)$. Being $\vf$ Lipschitz,
\[
|\vf(f(x))-\vf(y_0)|\le\, L\, d_Y(y_0,f(x))\text{ for each }x\in X\,,
\]
and we get the desired conclusion.
\\
(ii) Let $Z:=\,\{z_i:\,i\in\N\}\subset Y$ be a dense set. Let us denote
\[
\tilde f_\vf:=\,f_\vf-\vf(y_0):\,X\to\R \text{ and }u_i:=\vf(z_i)+L\,u_{z_i}:\,X\to\R\,,
\]
where $u_z$ is the section of $f$ defined in \eqref{sectfMS}.
Note that we can represent $f_\vf$ as
\begin{equation}\label{represvf}
\vf(f(x))=\,\inf_{i\in \N}u_i(x)\text{ for each }x\in X\,.
\end{equation}
Indeed, being $\vf$ Lipschitz,
\[|\varphi(f(x))-\varphi(z_i)|\leq L\,d_Y(z_i,f(x)).\]
This means that 
\[\varphi(f(x))\leq \varphi(z_i)+L\,d_Y(z_i,f(x))=\varphi(z_i)+L\,u_{z_i}(x)=u_i(x).\]
%
Moreover, consider a sequence $(i_k)_{k\in\N}$ such that $z_{i_k}\to f(x)$. We get
\[\begin{split}|u_{i_k}(x)-\varphi(f(x))|&=|\varphi(z_{i_k})+L\,d_Y(z_{i_k},f(x))-\varphi(f(x))|\\ &\leq |\varphi(z_{i_k})-\varphi(f(x))|+|L\,d_Y(z_{i_k},f(x))|\leq 2L\,d_Y(z_{i_k},f(x))\to 0,\end{split}\]which implies \eqref{represvf}.
Thus, from \eqref{represvf}, it also follows that
\begin{equation}\label{represtildefvf}
\tilde f_\vf(x)=\,\vf(f(x))-\vf(y_0)=\,\inf_{i\in \N}(\tilde u_i (x))\,.
\end{equation}
where $\tilde u_i:=\,u_i-\vf(y_0)$. Note also that $\tilde u_i\in D^{1,p}(X)\cap L^p_{\rm loc}(X,\mu)$ with 
\begin{equation}
\rho_{\tilde u_i}\le L g \,\mu\text{-a.e. in $X$ for each $i\in \N$} 
\end{equation}
where $\rho_{\tilde u_i}$ denotes the minimal $p$-weak upper gradient of $\tilde u_i$ and $g$ a $p$-weak upper gradient of $f$, that is a function satisfying  in \eqref{domugsect}. 
Let $v_j:=\,\inf_{1\le i\le j}\tilde u_i$. From  \cite[Proposition 7.1.8]{HKST}), $D^{1,p}(X)$ is a lattice, that is $v_j\in D^{1,p}(X)$ for each $j$ and
\begin{equation}\label{UGvjunifbded}
\rho_{v_j}\le  L \,g\quad\mu\text{-a.e. in $X$, for each $j$}\,.
\end{equation}
Now observe that, since
\begin{equation}\label{domvj}
\tilde f_\vf\le\, v_j\le\, v_1\quad\text{for each $j$} \,,
\end{equation}
from \eqref{represtildefvf} and \eqref{domvj}, by applying Lebesgue dominated convergence theorem, it follows that
\begin{equation}\label{vjtotildefvf}
v_j\to \tilde f_\vf\text{ in }L^p_{\rm loc}(X,\mu)\text{ as $j\to\infty$.}
\end{equation}
Let $x_0\in X$ and $\Omega_n:=\,U(x_0,n)$, if $n\in\N$. Then $(\Omega_n)_n$ is an increasing sequence of open balls of $X$ and $X=\cup_{n=1}^\infty \Omega_n$. From \eqref{vjtotildefvf} and \eqref{UGvjunifbded}, by the locality of Sobolev functions in $N^{1,p}(X)$ (see  \cite[Lemma 7.3.22]{HKST}), there exists 
\begin{equation}\label{tildefvfinN1p}
\hat f\in N^{1,p}(X)\text{  such that } \hat f=\,\tilde f_\vf \,\mu\text{-a.e. in } X
\end{equation} 
and $\hat f$ admits a $p$-weak upper gradient $\rho_\infty$  in $X$ satisfying
\begin{equation}\label{UGtildefvf}
\rho_\infty\le\, L\,g\quad\mu \text{-a.e. in }X.
\end{equation}
By \eqref{tildefvfinN1p} and \eqref{UGtildefvf}, it follows that $\tilde f_\vf$ admits a $\mu$-representative $\hat f\in N^{1,p}(X)$ with minimal $p$-weak upper gradient $\rho_{\hat f}\le\, L\,g$ $\mu$-a.e. in $X$. Thus we get the desired conclusion.
\end{proof}
\begin{proof}[Proof of Theorem \ref{W=N}]
Let us first show the  inclusion
\begin{equation*}
N^{1,p}(X:Y)\subset W^{1,p}(X,Y)\,.
\end{equation*}
Let $f\in N^{1,p}(X:Y)$. 
From Theorem \ref{charND},
\begin{equation}\label{finLp}
f\in N^{1,p}(X:Y)\subset\,D^{1,p}(X:Y)\subset D^{1,p}(X:\ell^\infty(Y))\,,
\end{equation}
and the map
\[
X\ni x\mapsto d_Y(y_0,f(x))\text{ is in }L^p(X,\mu)\,.
\]
For a given $y\in Y$, let $\vf_y:Y\to\R$ be the function, $\vf_y(z):=\,d_Y(y,z)$. Then by applying Theorem \ref{charsect} (ii)
\begin{equation}\label{CRf}
f_y:=\,\vf_y\circ f\in D^{1,p}(X)\text{ and }\rho_{f_y}\le\rho_{f}\,\text{$\mu$- a.e. in }X\text{, for each }y\in Y\,.
\end{equation}
Thus, combining \eqref{finLp} and \eqref{CRf}, it follows that $f\in W^{1,p}(X,Y)$.\\
Let us now show the inclusion
\begin{equation*}
W^{1,p}(X,Y)\subset N^{1,p}(X:Y)\,.
\end{equation*}
Let  $f\in W^{1,p}(X,Y)$. By definition,  we can can assume $X\ni x\mapsto d_Y(y_0,f(x))$ is in  $L^p(X,\mu)$. By applying Lemma \ref{CRlipW} and Theorem \ref{charsect} (ii), we can conclude that there exists  a $\mu$-representative of $f$ such that 
$f\in N^{1,p}(X:Y)$.
\end{proof}
We are going to  apply Theorem \ref{W=N} in the setting of  Carnot groups. In particular, for maps between Carnot groups, we have chosen to fix $y_0=0_{\G_2}$ (see Definition \ref{SobspacesvaluedCarnot}). Let us first recall the coincidence between the Newtonian Sobolev space and the horizontal Sobolev space of real valued  functions, which is well-known  (see, for instance, \cite[Proposition C.12]{KMX}).
\begin{theorem}\label{NG=WG} Let $1<\,p<\,\infty$ and  $\Omega\subset\G$ be  an open set of a Carnot group $\G\equiv\R^n$. Then
\[
N_\G^{1,p}(\Omega)=\,W_\G^{1,p}(\Omega)
\]
where $N_\G^{1,p}(\Omega)$ and $W_\G^{1,p}(\Omega)$  denote the Newtonian Sobolev space defined before between metric spaces $X=\,(\Omega,d_{c},\mathcal L^n)$ and $Y=\R$ and the horizontal Sobolev space defined in Definition \ref{horSob}, respectively.
\end{theorem}
As an immediate consequence of Theorems \ref{W=N} and \ref{NG=WG} and Definition \ref{SobspacesvaluedCarnot}, the following result follows (see also \cite[Proposition C.13]{KMX})
\begin{corollary}\label{metricvaluedNG=WG}
Let $\Omega\subset\G$ be  an open set of a Carnot group $\G\equiv\R^n$ and let $Y\equiv (Y,d_Y,y_0)$ be a pointed complete separable metric space. Then
\[
N^{1,p}(\Omega:Y)=\,W^{1,p}(\Omega,Y) 
\]
where $N^{1,p}(\Omega:Y)$ and $W^{1,p}(\Omega,Y)$ denote the Newtonian space between metric space $X=\,(\Omega,d_{c},\mathcal L^n)$ and $Y$, and the  metric space-valued Sobolev space defined by
sections introduced before, respectively.
\end{corollary}


\end{document}